\documentclass[a4paper,11pt]{amsart} 
\usepackage{amsmath, amsfonts, amssymb}
\usepackage[bookmarks=false]{hyperref}
\usepackage[shortlabels]{enumitem}

\usepackage{color}

\numberwithin{equation}{section}

\newtheorem{theoremstar}{Theorem}

\newtheorem{corstar}[theoremstar]{Corollary}

\newtheorem{theorem}{Theorem}[section]
\newtheorem{lemma}[theorem]{Lemma}
\newtheorem{corollary}[theorem]{Corollary}
\newtheorem{proposition}[theorem]{Proposition}

\theoremstyle{definition}
\newtheorem{remark}[theorem]{Remark}

\newtheorem{definition}[theorem]{Definition}

\newtheorem{problem}{Problem}

\newtheorem*{definition*}{Definition}
\newtheorem*{definitions*}{Definitions}

\newtheorem*{claim*}{Claim}

\newtheorem*{question*}{Question}

\newcommand{\R}{\mathbf{R}}
\newcommand{\C}{\mathbf{C}}
\newcommand{\Z}{\mathbf{Z}}
\newcommand{\F}{\mathbf{F}}

\newcommand{\N}{\mathbf{N}}
\newcommand{\B}{\mathbf{B}}

\newcommand{\T}{\mathbf{T}}

\newcommand{\cU}{\mathcal{U}}

\newcommand{\cR}{\mathcal{R}}

\newcommand{\Ad}{\operatorname{Ad}}
\newcommand{\id}{\text{\rm id}}

\newcommand{\Aut}{\operatorname{Aut}}

\newcommand{\rB}{\mathord{\text{\rm B}}}
\newcommand{\rZ}{\mathord{\text{\rm Z}}}
\newcommand{\rH}{\mathord{\text{\rm H}}}
\newcommand{\rL}{\mathord{\text{\rm L}}}
\newcommand{\rC}{\mathord{\text{\rm C}}}
\newcommand{\dom}{\mathord{\text{\rm dom}}}

\newcommand{\ran}{\mathord{\text{\rm ran}}}
\newcommand{\ap}{\mathord{\text{\rm ap}}}

\newcommand{\rE}{\mathord{\text{\rm E}}}

\newcommand{\Range}{\mathord{\text{\rm Range}}}
\newcommand{\Prob}{\mathord{\text{\rm Prob}}}

\newcommand{\Sd}{\mathord{\text{\rm Sd}}}
\newcommand{\SL}{\mathord{\text{\rm SL}}}

\newcommand{\core}{\mathord{\text{\rm c}}}
\newcommand{\rd}{\mathord{\text{\rm d}}}

\newcommand{\Stab}{\operatorname{Stab}}

\newcommand{\ovt}{\mathbin{\overline{\otimes}}}

\newcommand{\rS}{\mathord{\text{\rm S}}}

\newcommand{\Leb}{\mathord{\text{\rm Leb}}}

\newcommand{\vphi}{\varphi}

\newcommand{\I}{{\rm I}}
\newcommand{\II}{{\rm II}}
\newcommand{\III}{{\rm III}}

\title[Strongly ergodic equivalence relations]{Strongly ergodic equivalence relations: \\ spectral gap and type III invariants}

\begin{document}

\begin{abstract}
We obtain a spectral gap characterization of strongly ergodic equivalence relations on standard measure spaces. We use our spectral gap criterion to prove that a large class of skew-product equivalence relations arising from measurable $1$-cocycles with values into locally compact abelian groups are strongly ergodic. By analogy with the work of Connes on full factors, we introduce the Sd and $\tau$ invariants for type ${\rm III}$ strongly ergodic equivalence relations. As a corollary to our main results, we show that for any type ${\rm III_1}$ ergodic equivalence relation $\mathcal R$, the Maharam extension $\core(\mathcal R)$ is strongly ergodic if and only if $\mathcal R$ is strongly ergodic and the invariant $\tau(\mathcal R)$ is the usual topology on $\R$. We also obtain a structure theorem for almost periodic strongly ergodic equivalence relations analogous to Connes' structure theorem for almost periodic full factors. Finally, we prove that for arbitrary strongly ergodic free actions of bi-exact groups (e.g.\ hyperbolic groups), the Sd and $\tau$ invariants of the orbit equivalence relation and of the associated  group measure space von Neumann factor coincide.
\end{abstract}

\address{Laboratoire de Math\'ematiques d'Orsay\\ Universit\'e Paris-Sud\\ CNRS\\ Universit\'e Paris-Saclay\\ 91405 Orsay\\ FRANCE}

\author{Cyril Houdayer}
\email{cyril.houdayer@math.u-psud.fr}

\author{Amine Marrakchi}
\email{amine.marrakchi@math.u-psud.fr}

\author{Peter Verraedt}
\email{peter.verraedt@math.u-psud.fr}

\thanks{Research supported by ERC Starting Grant GAN 637601}

\subjclass[2010]{37A20, 37A40, 46L10, 46L36}

\keywords{Equivalence relations; Full groups; Maharam extension; Skew-product; Spectral gap; Strong ergodicity; von Neumann algebras}

\maketitle

\section{Introduction and statement of the main results}

Following \cite{Sc79}, a nonsingular action $\Gamma \curvearrowright (X, \mu)$ of a countable discrete group $\Gamma$ on a standard probability space $(X, \mu)$ is {\em strongly ergodic} if any sequence of $\Gamma$-almost invariant measurable subsets is trivial, i.e.\ for any sequence of measurable subsets $A_n \subset X$ such that $\lim_n \mu(A_n \triangle \gamma A_n) = 0$ for every $\gamma \in \Gamma$, we have $\lim_n \mu(A_n)(1 - \mu(A_n)) = 0$. Likewise, a nonsingular equivalence relation $\mathcal R$ on a standard probability space $(X, \mu)$ is {\em strongly ergodic} if for any sequence of measurable subsets $A_n \subset X$ such that $\lim_n \mu(A_n \triangle g A_n) = 0$ for every $g \in [\mathcal R]$, we have $\lim_n \mu(A_n)(1 - \mu(A_n)) = 0$. It is easy to see that the notion of strong ergodicity only depends on the measure class of $\mu$. By Rokhlin's Lemma, a hyperfinite or equivalently amenable (by \cite{CFW81}) equivalence relation on a diffuse standard probability space is never strongly ergodic (see \cite[Proposition 2.2]{Sc79}). In particular, a nonsingular action $\Gamma \curvearrowright (X, \mu)$ of an amenable countable discrete group on a diffuse standard probability space is never strongly ergodic.

The notion of strong ergodicity is related to the notion of fullness for von Neumann factors. Following \cite{Co74}, a factor $M$ with separable predual is {\em full} if any uniformly bounded centralizing sequence is trivial, i.e.\ for any uniformly bounded sequence $x_n \in M$ such that $\lim_n \|x_n \psi - \psi x_n\| = 0$ for every $\psi \in M_\ast$, there exists a bounded sequence $\lambda_n \in \C$ such that $x_n - \lambda_n 1 \to 0$ $\ast$-strongly. By the work of Connes \cite{Co74, Co75b} and Haagerup \cite{Ha85}, a hyperfinite or equivalently amenable diffuse factor with separable predual is never full. Following \cite{FM75}, denote by $\rL(\mathcal R)$ the factor associated with the ergodic equivalence relation $\mathcal R$ on the standard probability space $(X, \mu)$. If $\rL(\mathcal R)$ is full then $\mathcal R$ is strongly ergodic. The converse is however not true as demonstrated by Connes and Jones \cite{CJ81}.

For nonsingular equivalence relations $\mathcal R(\Gamma \curvearrowright X)$ arising from nonsingular actions of countable discrete group $\Gamma \curvearrowright (X, \mu)$ on standard probability spaces, strong ergodicity of the action $\Gamma \curvearrowright (X, \mu)$ is equivalent to strong ergodicity of the orbit equivalence relation $\mathcal R(\Gamma \curvearrowright X)$. As observed by Schmidt in \cite[Proposition 2.10]{Sc79}, for any countable discrete group $\Gamma$ and any probability measure preserving (pmp) action $\Gamma \curvearrowright (X, \mu)$, if the Koopman unitary representation $\pi : \Gamma \to \mathcal U(\rL^2(X, \mu) \ominus \C \mathbf 1_X)$ has {\em spectral gap} in the sense that $\pi$ does not weakly contain the trivial representation, then the orbit equivalence relation $\mathcal R(\Gamma \curvearrowright X)$ is strongly ergodic. The converse is however not true as explained in \cite[Example 2.7]{Sc80}. Any nonamenable countable discrete group $\Gamma$ admits a strongly ergodic free pmp action, namely the plain Bernoulli action $\Gamma \curvearrowright ([0, 1]^\Gamma, \Leb^{\otimes \Gamma})$. We refer the reader to \cite{HI15, Oz16} and the references therein for various examples of type ${\rm III}$ strongly ergodic free actions  associated with free groups (\cite{HI15}) and lattices in simple connected Lie groups with finite center (\cite{Oz16}).

Our first main result provides a spectral gap characterization of strongly ergodic equivalence relations of type ${\rm II_1}$ and of type ${\rm III_\lambda}$ for $\lambda \in (0, 1]$. Recall in that respect that a strongly ergodic equivalence relation can never be of type ${\rm III_0}$ (see e.g.\ \cite[Remark 2.9]{Sc79}). Before stating Theorem \ref{thmstar:spectral-gap}, we need to introduce some terminology. A nonsingular automorphism $\theta$ of a standard probability space $(X,\mu)$ is said to be \emph{$\mu$-bounded} if the function $\log \left( \frac{\mathrm{d} (\mu \circ \theta)}{\mathrm{d}\mu} \right)$ is bounded. In this case, the map $L_\theta : f \mapsto \theta(f)$ defines a bounded operator in $\B(\rL^{2}(X,\mu))$.

\begin{theoremstar}\label{thmstar:spectral-gap}
Let $\mathcal R$ be any ergodic equivalence relation on a standard probability space $(X, \mu)$. Assume that $\mathcal R$ either preserves $\mu$ or is of type ${\rm III}$. Then the following assertions are equivalent:
\begin{itemize}
\item [$(\rm i)$] $\mathcal R$ is strongly ergodic.
\item [$(\rm ii)$] There exist a constant $\kappa > 0$ and a finite family $\theta_1, \dots, \theta_n$ of $\mu$-bounded elements in $[\mathcal R]$ such that
\begin{equation}\label{eq:spectral-gap}
\forall f \in \rL^2(X, \mu), \quad \|f - \mu(f)\|_2^2 \leq \kappa \sum_{k = 1}^n \|\theta_k(f) - f\|_2^2.
\end{equation}
\end{itemize}
\end{theoremstar}

When $\mathcal R$ preserves $\mu$, the inequality in \eqref{eq:spectral-gap} shows that the Koopman unitary representation $\pi : [\mathcal R] \to \mathcal U(\rL^2(X, \mu) \ominus \C \mathbf 1_X)$ has spectral gap. Thus, in the pmp case and working inside the ambient full group $[\mathcal R]$, Theorem \ref{thmstar:spectral-gap} provides a converse to Schmidt's observation \cite[Proposition 2.10]{Sc79}. As we will see later, the main interest of Theorem \ref{thmstar:spectral-gap} is also to provide a spectral gap characterization for arbitrary strongly ergodic equivalence relations of type ${\rm III}$. Note that the inequality in \eqref{eq:spectral-gap} can be reformulated by saying that the eigenvalue $0$ is a simple isolated point in the spectrum of the positive selfadjoint bounded operator $\sum_{k = 1}^n |L_{\theta_k} - \mathbf 1|^2$. Let us point out that the spectral gap characterization of strongly ergodic equivalence relations in Theorem \ref{thmstar:spectral-gap} is analogous to Connes' spectral gap characterization of full factors of type ${\rm II_1}$ (see \cite[Theorem 2.1]{Co75b}). We refer the reader to \cite[Theorem A]{Ma16} and \cite[Theorem 3.2]{HMV16} for very recent spectral gap characterizations of full factors of type ${\rm III}$.

We next use our spectral gap criterion to prove that a large class of skew-product equivalence relations are strongly ergodic. Before stating Theorem \ref{thmstar:skew-product}, we need to introduce some more terminology. Let $\cR$ be any nonsingular equivalence relation on a standard measure space $(X, \mu)$, $G$ any locally compact second countable abelian group and $\Omega \in \rZ^1(\cR,G)$ any measurable $1$-cocycle. On $X \times G$, put the measure class obtained as the product of the measure class of $X$ and the Haar measure class of $G$. The {\em skew-product} of $\cR$ by $\Omega$ is the equivalence relation $\cR \times_\Omega G$ on $X \times G$ defined by 
\begin{align*}
((x,g), (y,h)) \in \cR \times_\Omega G \quad\text{if and only if}\quad (x, y) \in \cR \; \text{ and } \; \Omega(x,y) = gh^{-1}
\end{align*}
for all $x,y\in X$ and all $g,h\in G$. Let $\widehat{G}$ be the Pontryagin dual of $G$ and introduce the map $\widehat{\Omega} : \widehat{G} \rightarrow \rZ^{1}(\cR, \mathbf T)$ defined by $\widehat{\Omega}(p)(x,y)=\langle p , \Omega(x,y) \rangle$ for a.e.\ $(x,y) \in \cR$, where $\langle \, \cdot \, , \cdot \, \rangle : \widehat{G} \times G \rightarrow \T$ is the duality pairing. Note that $\widehat{\Omega}$ is a continuous group homomorphism. We also introduce the continuous homomorphism $[ \widehat{\Omega}] : \widehat{G} \rightarrow \rH^{1}(\cR, \mathbf T)$ which sends $p \in \widehat{G}$ to the cohomology class $[ \widehat{\Omega}(p)] \in \rH^{1}(\cR, \mathbf T)$.

\begin{theoremstar}\label{thmstar:skew-product}
Let $\cR$ be any ergodic equivalence relation on a standard measure space $X$. Let $\Omega \in \rZ^1(\cR,G)$ be any measurable $1$-cocycle with values into a locally compact second countable abelian group $G$. Consider the following assertions:
\begin{itemize}
\item [$(\rm i)$] The skew-product equivalence relation $\cR \times_\Omega G$ is strongly ergodic.
\item [$(\rm ii)$] The equivalence relation $\cR$ is strongly ergodic and the map $[\widehat{\Omega}] : \widehat{G} \rightarrow \rH^1(\cR, \mathbf T)$ is a homeomorphism onto its range.
\end{itemize}
Then $(\rm i) \Rightarrow (\rm ii)$ and if $G$ contains a lattice, we also have $(\rm ii) \Rightarrow (\rm i)$.
\end{theoremstar}

It is unclear whether the assumption that $G$ contains a lattice (i.e.\ a discrete cocompact subgroup) is really needed but it already covers all the most common cases: compact groups, discrete groups, connected groups and all their direct products.

\begin{problem}
Prove that the equivalence $(\rm ii) \Leftrightarrow (\rm i)$ holds for all locally compact abelian groups.
\end{problem}

We also point out the fact that $\rZ^{1}(\cR,\T)$ can be identified with the group $\Aut(\rL(\cR)/\rL^{\infty}(X))$ of all automorphisms of $\rL(\cR)$ that fix $\rL^{\infty}(X)$ pointwise (see Lemma \ref{lemma: link}). Hence the homomorphism $ \widehat{\Omega}$ defines an action of $\widehat{G}$ on $\rL(\cR)$ and this allows us to identify $\rL(\cR \times_\Omega G)$ with the crossed product $\rL(\cR) \rtimes_{\widehat{\Omega}} \widehat{G}$. Therefore, in the case when $G$ is compact (or equivalently $\widehat{G}$ is discrete), Theorem \ref{thmstar:skew-product} can be seen as an analogue of a theorem of Jones on fullness of crossed products \cite{Jo81} and its recent generalization to arbitrary factors \cite[Theorem B]{Ma16}. These results inspired Theorem \ref{thmstar:skew-product}.

We now apply Theorem \ref{thmstar:skew-product} to the structure of type ${\rm III}$ strongly ergodic equivalence relations and we introduce two new orbit equivalence invariants, namely the Sd and $\tau$ invariants. Let $\mathcal R$ be any type ${\rm III}$ ergodic equivalence relation on a standard measure space $(X, \mu)$. Denote by $\delta_\mu \in \rZ^1(\mathcal R, \R^+_0)$ the Radon-Nikodym $1$-cocycle. Its cohomology class $[\delta_\mu] \in \rH^1(\mathcal R, \R^+_0)$ only depends on the measure class of $\mu$. The skew-product equivalence relation $\core(\mathcal R) = \mathcal R \times_{\delta_\mu} \R^+_0$ is called the {\em Maharam extension} of $\mathcal R$. Then $\core(\mathcal R)$ is an infinite measure preserving equivalence relation that only depends on the measure class of $\mu$ up to canonical isomorphism. Moreover, $\mathcal R$ is of type ${\rm III_1}$ if and only if $\core(\mathcal R)$ is ergodic. Let us now assume that $\mathcal R$ is moreover strongly ergodic. Then $\rB^1(\mathcal R, \mathbf T) \subset \rZ^1(\mathcal R, \mathbf T)$ is a closed subgroup and therefore $\rH^1(\mathcal R, \mathbf T)$ is a Hausdorff Polish group (see \cite[Proposition 2.3]{Sc79}). By analogy with the work of Connes on full factors of type ${\rm III}$ \cite{Co74}, we introduce the invariant $\tau(\mathcal R)$ as the weakest topology on $\R$ that makes the map $[\widehat{\delta_\mu}] : \R \to \rH^1(\mathcal R, \mathbf T) : t \mapsto [\delta_\mu^{{\bf i}t}]$ continuous. Applying Theorem \ref{thmstar:skew-product} to the locally compact abelian group $G = \R^+_0$ and to the Radon-Nikodym $1$-cocycle $\Omega = \delta_\mu \in \rZ^1(\mathcal R, \R^+_0)$, we obtain the following characterization of strong ergodicity for the Maharam extension of any type ${\rm III_1}$ ergodic equivalence relation.

\begin{corstar}\label{corstar:maharam}
Let $\cR$ be any type $\III_1$ ergodic equivalence relation on a standard measure space. Then the Maharam extension $\core(\cR)$ is strongly ergodic if and only if $\cR$ is strongly ergodic and $\tau(\cR)$ is the usual topology on $\R$.
\end{corstar}

Let us point out that Corollary \ref{corstar:maharam} is analogous to \cite[Theorem C]{Ma16} where it was shown that for any factor $M$ with separable predual, the continuous core $\core(M)$ is full if and only if $M$ is full and $\tau(M)$ is the usual topology on $\R$.

We next turn to investigating a large class of type ${\rm III}$ strongly ergodic equivalence relations for which the $\tau$ invariant is not the usual topology on $\R$. Let $\mathcal R$ be any type ${\rm III}$ ergodic equivalence relation on a standard measure space $(X, \mu)$. By analogy with the work of Connes on almost periodic factors of type ${\rm III}$ \cite{Co74}, we say that $\mathcal R$ is {\em almost periodic} if there exists a $\sigma$-finite measure $\nu$ on $X$ that is equivalent to $\mu$ and for which the Radon-Nikodym $1$-cocycle $\delta_\nu \in \rZ^1(\mathcal R, \R^+_0)$ essentially takes values into a countable subset of $\R^+_0$ that we denote by $\Range(\delta_\nu)$. Such a measure $\nu$ is called \emph{almost periodic} and the set of all almost periodic measures is denoted by $\mathcal{M}_{\ap}(X,\cR)$. We then introduce the invariant $\Sd(\mathcal R)$ as the intersection of $\Range(\delta_\nu)$ over all measures $\nu \in \mathcal{M}_{\ap}(X,\cR)$. The invariant $\Sd(\mathcal R)$ is a countable subgroup of $\R_0^+$ (see Proposition \ref{proposition: group}). Our next main result provides a structure theorem for type ${\rm III}$ strongly ergodic almost periodic equivalence relations  which is analogous to Connes's structure theorem for almost periodic full factors of type ${\rm III}$ \cite{Co74}.

\begin{theoremstar}\label{thmstar:almost-periodic}
Let $\mathcal R$ be a type ${\rm III}$ strongly ergodic almost periodic equivalence relation on a standard measure space $X$. Put $\Gamma=\mathrm{Sd}(\cR)$. Then the following assertions hold true:
\begin{itemize}
\item [$(\rm i)$] If $\cR$ is of type $\III_\lambda$ for $\lambda \in (0,1)$, then $\Gamma=\{\lambda^{n} \mid n \in \Z \}$. If $\cR$ is of type $\III_1$, then $\Gamma$ is a dense countable subgroup of $\R^{+}_0$.
\item [$(\rm ii)$] For any sequence $(t_n)_{n \in \N}$ in $\R$, we have $t_n \to 0$ with respect to $\tau(\mathcal R)$ if and only if $\gamma^{{\bf i}t_n} \to 1$ for every $\gamma \in \Gamma$.
\item [$(\rm iii)$] There exists a measure $\nu \in \mathcal{M}_{\ap}(X,\cR)$ such that $\Range(\delta_\nu) = \Gamma$.
\item [$(\rm iv)$] For any measure $\nu$ as in $(\rm iii)$, the measure preserving subequivalence relation $\mathcal R_\nu \subset \mathcal R$ defined by $\mathcal R_\nu := \ker(\delta_\nu)$ is strongly ergodic.
\item [$(\rm v)$] For any infinite measures $\nu_1, \nu_2$ as in $(\rm iii)$, there exist $\alpha > 0$ and $\theta \in [\mathcal R]$ such that $\theta_\ast \nu_1=\alpha \nu_2 $.

\end{itemize}
\end{theoremstar}

In Section \ref{section:computations}, we use {\em generalized Bernoulli} equivalence relations to construct families of strongly ergodic equivalence relations with prescribed Sd and $\tau$ invariants. For the class of plain Bernoulli equivalence relations arising from nonamenable groups, the corresponding factor is full and the Sd and $\tau$ invariants of the equivalence relation and of the factor coincide. However, we also construct a family of generalized Bernoulli equivalence relations with prescribed Sd and $\tau$ invariants and for which the corresponding factor is not full. The following open question is interesting in this respect.

\begin{problem}
Construct an example of an ergodic equivalence relation $\mathcal R$ on a standard measure space $X$ such that the associated factor $\rL(\mathcal R)$ is almost periodic but $\cR$ is not almost periodic. What if we require $\rL(\mathcal R)$ to be full?
\end{problem}

As we already pointed out, the Sd and $\tau$ invariants of a strongly ergodic equivalence relation $\mathcal R$ may differ from the Sd and $\tau$ invariants of the corresponding factor $\rL(\mathcal R)$. Using \cite[Theorems A]{HI15}, we obtain the following rigidity result for arbitrary strongly ergodic free actions $\Gamma \curvearrowright (X, \mu)$ of bi-exact countable discrete groups (e.g. hyperbolic groups) on standard measure spaces: the Sd and $\tau$ invariants of the orbit equivalence relation $\mathcal R(\Gamma \curvearrowright X)$ and of the factor $\rL(\mathcal R(\Gamma \curvearrowright X))$ coincide (note that $\rL(\mathcal R(\Gamma \curvearrowright X))$ is full by \cite[Theorem C]{HI15}). We refer to \cite[Definition 15.1.2]{BO08} and section \ref{section:computations} for the definition of bi-exact groups.

\begin{theoremstar}\label{thmstar:bi-exact}
Let $\Gamma$ be any bi-exact countable discrete group and $\Gamma \curvearrowright (X, \mu)$ any strongly ergodic free action on a standard measure space. Then $\rL(\mathcal R(\Gamma \curvearrowright X))$ is a full factor and $$\tau \left(\mathcal R(\Gamma \curvearrowright X) \right) = \tau \left(\rL(\mathcal R(\Gamma \curvearrowright X)) \right).$$
If moreover $\mathcal R(\Gamma \curvearrowright X)$ is almost periodic, then $\rL(\mathcal R(\Gamma \curvearrowright X))$ is almost periodic and 
$$\Sd \left(\mathcal R(\Gamma \curvearrowright X) \right) = \Sd \left(\rL(\mathcal R(\Gamma \curvearrowright X)) \right).$$
\end{theoremstar}

In \cite[Section 6]{HI15}, for every $\lambda \in (0, 1]$, a concrete example of a strongly ergodic free action $\F_\infty \curvearrowright (X_\lambda, \mu_\lambda)$ on a standard measure space was constructed using the main result of \cite{BISG15}. The following open question is interesting in this respect.

\begin{problem}\label{problem3}
Construct examples of strongly ergodic free actions of free groups $\F_n \curvearrowright (X, \mu)$  on a standard measure space for which the induced orbit equivalence relation $\mathcal R(\F_n \curvearrowright X)$ (and hence the factor $\rL(\mathcal R(\F_n \curvearrowright X))$ by Theorem \ref{thmstar:bi-exact}) has prescribed Sd and $\tau$ invariants.
\end{problem}

\subsection*{Added in the proof} Since this paper was posted on the arXiv in April 2017, Problem \ref{problem3} above has been solved in \cite{VW17}. Indeed, it is shown in \cite[Section 7]{VW17} that for a large class of nonsingular Bernoulli actions of the free groups $\F_n$ with $n \geq 3$, the induced orbit equivalence relation is strongly ergodic and has prescribed Sd and $\tau$ invariants.

\subsection*{Acknowledgments} It is our pleasure to thank Damien Gaboriau for his valuable comments.

\tableofcontents

\section{Preliminaries}

\subsection{Notations}
Let $(X,\mathcal{B},\mu)$ be a standard $\sigma$-finite measure space. We will often omit the $\sigma$-algebra $\mathcal{B}$ as well as the measure $\mu$ when considering objects that only depend on the measure class of $\mu$. In particular, for every Polish space $E$, we denote by $\rL^{0}(X,E)$ the set of all equivalence classes of measurable functions from $X$ to $E$, for the equivalence relation of equality almost everywhere. Then $\rL^{0}(X,E)$ is a Polish space for the topology of convergence in measure (which does not depend on the choice of the measure within the same measure class). When $E=\C$, we will use the shorthand notation $\rL^0(X)=\rL^0(X,\C)$. 

If $f \in \rL^{0}(X,E)$, then the measure class of the push-forward measure $f_*\mu$ is well-defined and only depends on the measure class of $\mu$. The topological support of $f_*\mu$ (i.e.\ the smallest closed subset of $E$ on which it is supported) is called the \emph{essential range} of $f$ and is denoted $\overline{\mathrm{Range}}(f)$. If $f_*\mu$ is atomic, we say that $f$ is a step function and in this case, we denote by $\mathrm{Range}(f)$ the set of atoms of $f_*\mu$. Note that in this case, $\mathrm{Range}(f)$ is a countable subset and its closure is precisely $\overline{\mathrm{Range}}(f)$, by definition. 

Note that $\rL^0(X,\{0,1\})$ can be identified with the set $\mathfrak{P}(X)$ of all equivalence classes of measurable subsets of $X$. We will often abuse notations and identify a measurable subset $A \subset X$ with its equivalence class in $\mathfrak{P}(X)$. We will need the crucial fact that the boolean algebra $\mathfrak{P}(X)$ is complete: every increasing net has a supremum.

We will denote by $\mathcal{M}(X)$ the set of all $\sigma$-finite measures in the measure class of $X$ (note that this is smaller than the usual set of all Borel measures on $X$ which is also often denoted by $\mathcal{M}(X)$ in the literature). If $\mu,\nu \in \mathcal{M}(X)$, then their Radon-Nikodym derivative is a well-defined element $\frac{\rd \mu}{\rd \nu} \in \rL^{0}(X,\R_0^{+})$.

\subsection{Equivalence relations, full groups and strong ergodicity} Let $X$ be a standard measure space. Let $\cR$ be an equivalence relation on $X$. In this article, we will always assume that $\cR$ is measurable as a subset of $X \times X$, that it has countable classes and that it is \emph{nonsingular}, i.e.\ the $\cR$-saturation of every null-set of $X$ is again a null-set. As a measurable space, $\cR$ is equipped with a canonical measure class for which a subset $A \subset \cR$ is a null-set if and only if one of its coordinate projections on $X$ is a null-set. We refer the reader to \cite{FM75} for general background on nonsingular equivalence relations.

We define the \emph{full pseudo-group} of $\cR$, denoted by $[[\cR]]$, as the set of all partial isomorphisms $\theta : \dom(\theta) \rightarrow \ran (\theta)$ where $\dom (\theta),\ran (\theta) \in \mathfrak{P}(X)$ such that $(x, \theta(x)) \in \cR$ for a.e.\ $x \in \dom(\theta)$. If $\theta, \theta' \in [[\cR]]$, then we can naturally compose them to obtain a new element $\theta' \circ \theta \in [[ \cR]]$ with $\dom(\theta' \circ \theta)=\theta^{-1}(\dom(\theta') \cap \ran(\theta))$ and $\ran(\theta' \circ \theta)=\theta'(\ran(\theta) \cap \dom(\theta'))$. The set of invertible elements of $[[\cR]]$ is called the \emph{full group} of $\cR$ and is denoted by $[\cR]$. 

For every element $\theta \in [[\cR]]$, we let $\mathrm{graph}(\theta)=\{ (x, \theta(x)) \mid x \in \dom(\theta) \} \in \mathfrak{P}(\cR)$ be the graph of $\theta$. Then we have
\[ \cR = \bigcup_{ \theta \in [\cR]} \mathrm{graph}(\theta) \]
and any subset $A \in \mathfrak{P}(\cR)$ can be written as a disjoint union
\[ A=\bigsqcup_{n \in \N} \mathrm{graph}(\theta_n) \]
for some elements $\theta_n \in [[ \cR]]$. We have that $[\cR]$ is a Polish group for the topology  induced by the following metric
\[ \rd_\mu (\theta,\theta'):=\mu \left( \{ x \in X \mid \theta(x) \neq \theta'(x) \} \right) \]
where $\mu \in \mathcal{M}(X)$ is any probability measure (the induced topology does not depend on $\mu$).

Let $E$ be a Polish space. For every function $f \in \rL^0(X,E)$, we define two functions $f^\ell$ and $f^r$ in $\rL^0(\cR,E)$ by
\begin{align*}
 f^\ell(x,y) &=f(x),\\
 f^r(x,y) &=f(y) \quad \text{for a.e. } (x,y) \in \cR.
\end{align*}
We say that $f$ is invariant by $\cR$ if $f^\ell=f^r$, i.e.\ we have $f(x)=f(y)$ for a.e.\ $(x,y) \in \cR$. This is equivalent to the property that $\theta(f)=f$ for every $\theta \in [\cR]$, where $\theta(f):=f \circ \theta^{-1}$. We denote by $\rL^{0}(X,E)^{\cR}$ the subset of all $\cR$-invariant functions. We say that $\cR$ is \emph{ergodic} when $\rL^\infty(X)^\cR=\C$. We say that a sequence $f_n \in \rL^\infty(X), \; n \in \N$ is \emph{almost $\cR$-invariant} if $f_n^\ell-f_n^r \to 0$ in the measure topology. This property is equivalent to $\theta(f_n)-f_n \to 0$ for all $\theta \in [\cR]$.  An almost $\cR$-invariant sequence $(f_n)_{n \in \N}$ is called \emph{trivial} if there exists a sequence $\lambda_n \in \C$ such that $f_n-\lambda_n \to 0$ in the measure topology. We say that $\cR$ is \emph{strongly ergodic} if every almost $\cR$-invariant sequence in $\rL^\infty(X)$ is trivial.

\subsection{First cohomology of equivalence relations}
Let $\cR$ be an equivalence relation on a standard  measure space $X$ and $G$ a locally compact second countable abelian group. Put $\cR^{(2)} = \left\{(x,y,z)\in X^3 \mid (x,y)\in\cR \text{ and } (y,z)\in\cR \right\}$. As a measurable subset of $X^3$, $\cR^{(2)}$ is again equipped with a canonical measure class for which a subset $A \subset \cR^{(2)}$ is a null-set if and only if one (hence any) of its coordinate projections is a null-set.

A measurable $G$-valued $1$-{\em cocycle} for $\cR$ is a function $\Omega \in \rL^{0}(\cR,G)$ satisfying $\Omega(x,y)\Omega(y,z) = \Omega(x,z)$ for a.e.\ $(x,y,z) \in \cR^{(2)}$. We denote by $\rZ^{1}(\cR,G)$ the set of all such cocycles. It is a closed subgroup of the Polish abelian group $\rL^{0}(\cR,G)$.

\begin{remark} \label{rem_cocycle}
Given any $1$-cocycle $\Omega \in \rZ^{1}(\cR,G)$, one may always choose a representing $1$-cocycle such that  $\Omega(x,y)\Omega(y,z) = \Omega(x, z)$ holds for {\em every}  $(x, y, z) \in \mathcal R^{(2)}$ (see e.g.\ \cite[Theorem B.9]{Zi84} and \cite[Theorem 1]{FM75}).
\end{remark}

We define a boundary map $\partial : \rL^{0}(X,G) \ni f \mapsto \partial f  \in \rZ^{1}(\cR,G)$ by $(\partial f)(x,y)=f(x)f(y)^{-1}$ for a.e.\ $(x,y) \in \cR$. Note that $\partial$ is a continuous group homomorphism. The image of $\partial$ is denoted by $\rB^{1}(\cR,G)$ and its elements are called $1$-\emph{coboundaries}. The \emph{$1$-cohomology of $\cR$ with coefficients in $G$} is the quotient group $\rH^1(\cR,G) = \rZ^1(\cR,G)/\rB^1(\cR,G)$ equipped with the quotient topology (which is not necessarily Hausdorff). If $\Omega \in \rZ^1(\cR,G)$ is a $1$-cocycle, we denote by $[\Omega] \in \rH^1(\cR,G)$ its cohomology class.

\begin{proposition}[{\cite[Proposition 2.3]{Sc79}}]
 Let $\cR$ be an equivalence relation on a standard measure space $X$. Then $\cR$ is strongly ergodic if and only if $\rB^1(\cR,\T)$ is closed in $\rZ^1(\cR,\T)$. In that case, $\rH^1(\cR,\T)$ is a Hausdorff Polish group.
\end{proposition}
 
 From now on, if $G = \T$, we will use the shorthand notations $\rZ^1(\cR) = \rZ^1(\cR,\T)$, $\rB^1(\cR) = \rB^1(\cR,\T)$ and $\rH^1(\cR) = \rH^1(\cR,\T)$. 

 \begin{definition}
Let $\cR$ and $\mathcal{S}$ be two equivalence relations on two standard measure spaces $X$ and $Y$. A \emph{morphism} from $\cR$ to $\mathcal{S}$ is a nonsingular measurable map $f : X \rightarrow Y$ such that $(f(x), f(y)) \in \mathcal S$ for a.e.\ $(x,y) \in \cR$. Two such morphisms $f$ and $g$ are said to be \emph{equivalent} if $(f(x), g(x)) \in \mathcal S$ for a.e.\ $x \in X$.
\end{definition}

It is easy to see that if $f : X \rightarrow Y$ is a morphism from $\cR$ to $\mathcal{S}$, then the natural map $f^{*} : \rZ^1(\mathcal{S}) \ni \Omega \mapsto \Omega \circ (f \times f) \in \rZ^1(\cR)$ is a continuous group homomorphism which sends $1$-coboundaries to $1$-coboundaries. Hence it induces a continuous group homomorphism $[f^*] : \rH^{1}(\mathcal{S}) \rightarrow \rH^1(\cR)$. The next proposition shows that this induced map $[f^{*}]$ only depends on the equivalence class of $f$.

\begin{proposition}
\label{endo_homology}
Let $\cR$ and $\mathcal{S}$ be two equivalence relations on two standard measure spaces $X$ and $Y$. And let $f,g :  X \rightarrow Y$ be two morphisms from $\cR$ to $\mathcal{S}$. If $f$ and $g$ are equivalent, then $[f^*]=[g^*]$.
\end{proposition}
\begin{proof}
Let $c \in \rZ^{1}(\mathcal{S})$. Since $f \sim g$, we can define an element $u \in \rL^{0}(X,\T)$ by
\[ u(x)=c(f(x), g(x)) \quad \text{for a.e.\ } x \in X.  \]
Then, by using the cocycle identity, we compute
\[\partial u(x,y)= c(f(x),g(x)) \cdot \overline{c(f(y),g(y))}=c(f(x),f(y)) \cdot \overline{c(g(x),g(y))} \]
for a.e.\ $(x,y) \in \cR$ and this shows exactly that $[f^*(c)]=[g^*(c)]$.
\end{proof}

\begin{corollary}
\label{corner_homology}
Let $\cR$ be an ergodic equivalence relation on a standard measure space $X$. Let $Y \subset X$ be any nonzero measurable subset and let $\iota : Y \rightarrow X$ be the inclusion map. Then, the map $[\iota^{*}] : \rH^{1}(\cR) \rightarrow \rH^{1}(\cR_Y)$ is an isomorphism of topological groups.
\end{corollary}
\begin{proof}
Since $\cR$ is ergodic, we can find a family $(\theta_i)_{i \in I}$ of elements of the full pseudogroup of $\cR$ such that $\ran (\theta_i) \subset Y$ for all $i \in I$ and the family $(\dom (\theta_i))_{i \in I}$ forms a partition of $X \setminus Y$. Define a map $r : X \rightarrow Y$ by $r(x)=x$ if $ x \in Y$ and $r(x)=\theta_i(x)$ if $x \in \dom( \theta_i)$. By construction, $r$ is a morphism from $\cR$ to $\cR_Y$. Moreover, we have $r \circ \iota=\id_Y$ and $\iota \circ r \sim \id_X$ (i.e.\ $r$ is a \emph{retraction}). Hence, at the cohomological level, $[r^*]$ is an inverse of $[\iota^*]$.
\end{proof}

\subsection{Skew-product equivalence relations}
Let $\cR$ be an equivalence relation on a standard measure space $X$, $G$ a locally compact second countable abelian group and $\Omega \in \rZ^1(\cR,G)$ a $1$-cocycle. On $X \times G$, put the measure class obtained as the product of the measure class of $X$ and the Haar measure class of $G$. The skew-product of $\cR$ by $\Omega$ is the equivalence relation $\cR \times_\Omega G$ on $X \times G$ defined by 
$$((x,g),(y,h)) \in \cR \times_\Omega G  \quad\text{if and only if}\quad (x,y) \in \cR \; \text{ and } \; \Omega(x,y) = gh^{-1}$$
for all $x,y\in X$ and all $g,h\in G$.

Let $\Omega' \in \rZ^{1}(\cR,G)$ be another $1$-cocycle in the same cohomology class of $\Omega$. Take $\alpha \in \rL^{0}(X,G)$ a function such that $\Omega'=(\partial \alpha) \Omega$. Define an element $T_\alpha \in \Aut(X \times G)$ by
\[ T_\alpha(x,g)=(x,\alpha(x)g) \quad \text{for a.e.\ } (x,g) \in X \times G. \]
Then $T_\alpha$ is an isomorphism from $\cR \times_\Omega G$ to $\cR \times_{\Omega'} G$.

\subsection{Maharam extension and type classification of equivalence relations}
Let $\cR$ be an equivalence relation on a standard measure space $X$. Let $\mu \in \mathcal{M}(X)$. On $\cR \subset X\times X$, define two measures $\mu_\ell, \mu_r \in \mathcal{M}(\cR)$ as follows:
\begin{align*}
\mu_\ell(W) &= \int_X \sharp \left\{y \in X \mid (x,y)\in W \right\} \rd\mu(x),\\
\mu_r(W) &= \int_X \sharp \left\{x \in X \mid (x,y)\in W \right\} \rd\mu(y) \quad\text{for $W \in \mathfrak{P}(\cR)$.}
\end{align*}

We say that $\mu$ is $\cR$-invariant if $\mu_\ell=\mu_r$. In general, one can define the \emph{modulus} of $\mu$ (with respect to $\cR$) by 
\[ \delta_\mu := \frac{\rd\mu_\ell}{\rd\mu_r} \in \rZ^1(\cR,\R^+_0).\]
If $\nu \in \mathcal{M}(X)$ is another measure, then we have
\[ \delta_\nu \delta_\mu^{-1}= \partial \left( \frac{\rd\nu}{\rd\mu} \right) \in \rB^1(\cR,\R^+_0). \]
Hence we have a canonical cohomology class
\[ \delta:=[\delta_\mu] \in \rH^1(\cR,\R^+_0) \]
and the skew-product equivalence relation $\core(\cR)=\cR \times_{\delta_\mu} \R^+_0$ does not depend on the choice of $\mu$ up to canonical isomorphism. We call $\core(\cR)$ the \emph{Maharam extension} of $\cR$.

Now, suppose that the equivalence relation $\cR$ is ergodic. We say that $\cR$ is of type
\begin{itemize}[leftmargin=1.5cm,labelwidth=1cm,align=parleft]
\item[$\I$] if $\cR$ has only one equivalence class, up to measure zero;
\item[$\II_1$] if $\cR$ is not of type $\I$ and if there exists a probability measure $\mu \in \mathcal{M}(X)$ that is $\cR$-invariant;
\item[$\II_\infty$] if $\cR$ is not of type $\I$ and if there exists an infinite measure $\mu \in \mathcal{M}(X)$ that is $\cR$-invariant;
\item[$\III$] otherwise.
\end{itemize}

By analogy with \cite{Co72} (see also \cite{Kr67, Kr75}), for any ergodic equivalence relation $\mathcal R$, we define the S invariant by the formula
\[ \rS(\cR)=\bigcap_{ \mu \in \mathcal{M}(X)} \overline{\text{Range}}(\delta_\mu), \]
where $\overline{\text{Range}}(\delta_\mu)$ denotes the essential range of $\delta_\mu$ in $\R^+$. The S invariant can be computed by using a single measure $\mu \in \mathcal{M}(X)$ thanks to the following formula:
\[ \rS(\cR)=\bigcap_{ U \in \mathfrak{P}(X), \; U \neq \emptyset} \overline{\text{Range}}(\delta_{\mu}|_{ \cR_U}), \]
where $\cR_U=\cR \cap ( U \times U)$ and $\delta_{\mu}|_{ \cR_U}$ is the restriction of $\delta_\mu$ to $\cR_U$. If moreover, the $\mu$-preserving subequivalence relation $\cR_\mu:=\ker(\delta_\mu) \subset \cR$ is ergodic, then we actually have $$\mathrm{S}(\cR)=\overline{\text{Range}}(\delta_\mu).$$

Note that $\mathcal R$ is of type ${\rm I}$ or ${\rm II}$ if and only if $\rS(\mathcal R) = \{ 1 \}$.
Assume now that $\cR$ is a type ${\rm III}$ ergodic equivalence relation. Consider the translation action of $\R^{+}_0$ on $X \times \R^{+}_0$ and note that this action preserves $\core(\cR)$. Hence it induces an action of $\R^+_0$ on $\rL^\infty(X \times \R^{+}_0)^{\core(\cR)}$. Then $\rS(\cR) \cap \R^+_0$ is the kernel of this action and thus $\rS(\mathcal R) \cap \R^+_0$ is a closed subgroup of $\R^{+}_0$ (see \cite[Theorem XIII.2.23]{Ta03b}). We say that $\cR$ is of type
\begin{itemize}[leftmargin=3cm,labelwidth=2.5cm,align=parleft]
\item[$\III_0$] if $\rS(\cR)=\{0, 1\}$;
\item[$\III_\lambda, \; \lambda \in (0,1)$] if $\rS(\cR)= \{0\} \cup \lambda^{\Z}$;
\item[$\III_1$] if $\rS(\cR)=\R^{+}$.
\end{itemize}

Finally, by analogy with \cite{Co74}, we introduce the $\tau$ invariant:

\begin{definition}
 Let $\cR$ be a strongly ergodic equivalence relation on a standard measure space $X$. The \emph{$\tau$ invariant} of $\cR$, denoted by $\tau(\cR)$, is defined as the weakest topology on $\R$ that makes the map $\R \to \rH^1(\cR) : t \mapsto [\delta_\mu^{\mathbf{i}t}]$ continuous.
\end{definition}

\subsection{Von Neumann algebras associated with equivalence relations}\label{subsection: von Neumann}

Let $\cR$ be an equivalence relation on a standard measure space $X$. The von Neumann algebra associted with $\cR$, denoted by $\rL(\cR)$, is generated by a copy of $\rL^\infty(X)$ and a set of unitaries $u_\theta, \; \theta \in [\cR]$ which satisfy the following two conditions:
\begin{itemize}
\item $u_\theta u_\phi=u_{ \theta \circ \phi}$ for all $\theta, \phi \in [\cR]$. 
\item $u_\theta f u_\theta^*=\theta(f)$ for all $\theta \in [\cR]$ and $f \in \rL^\infty(X)$, where $\theta(f)=f\circ \theta^{-1}$.
\item There exists a faithful normal conditional expectation $$\rE_{\rL^\infty(X)} : \rL(\cR) \rightarrow \rL^\infty(X) $$ that is characterized by $\rE_{\rL^\infty(X)}(u_\theta)=1_{\{ x \in X \mid \theta(x)=x \} }$
for every $\theta \in [\cR]$.
\end{itemize}

Pick $\mu \in \mathcal{M}(X)$. 
We define $\rL^2(\cR, \mu_r)$ to be the Hilbert space of quadratic integrable functions with respect to $\mu_r$. Then $\rL(\cR)$ is naturally represented on $\rL^2(\cR,\mu_r)$. For each $\theta  \in [\cR]$, the unitary $u_\theta$ acts on $\rL^2(\cR,\mu_r)$ by
\begin{align*}
 (u_\theta\xi)(x,y) =\xi(\phi^{-1}(x),y) \quad\text{ for }\xi \in \rL^2(\cR,\mu_r), (x,y)\in \cR,
\end{align*}
and $\rL^\infty(X)$ acts on $\rL^2(\cR)$ by $(f\xi)(x,y) = f(x)\xi(x,y)$ for $f \in \rL^\infty(X), \xi \in \rL^2(\cR,\mu_r), (x,y)\in \cR$.

On $\rL(\cR)$, define the natural faithful normal semifinite weight $\vphi = \tau_\mu \circ \rE_{\rL^\infty(X)}$, where $\tau_\mu = \int_X \cdot \, \rd \mu$. Then $\rL^2(\cR,\mu_r)$ can be identified with the GNS construction of $\rL(\cR)$ with respect to $\varphi$. The modular operator $\Delta_\vphi$ of $\vphi$ is then determined by $(\Delta_\vphi^{\mathbf{i}t} \xi)(x,y) = \delta_\mu(x,y)^{\mathbf{i}t} \xi(x,y)$, for $t \in \R$, $\xi \in \rL^2(\cR, \mu_r)$, $(x,y)\in \cR$.

We recall the following two well-known facts. For an inclusion of von Neumann algebras $A\subset M$, we denote by $\Aut(M/A)$ the group of automorphisms of $M$ that fix $A$ pointwise.

\begin{lemma}\label{lemma: link}
 Let $\cR$ be a nonsingular equivalence relation on a standard measure space $X$. The following statements hold true.
 \begin{itemize}
  \item For every $1$-cocycle $c \in \rZ^1(\cR) \subset \rL^\infty(\cR)$, the map $\Ad(c) : \rL(\cR) \to \mathbf B(\rL^2(\cR))$ is an automorphism of $\rL(\mathcal R)$ that fixes $\rL^\infty(X)$ pointwise, i.e.\ $\Ad(c) \in \Aut(\rL(\cR)/\rL^\infty(X))$.
  \item The map $c \mapsto \Ad(c)$ is a topological isomorphism between $\rZ^1(\cR)$ and $\Aut(\rL(\cR)/\rL^\infty(X))$ that sends $\rB^1(\cR)$ onto $\Ad(\cU(\rL^\infty(X)))$.
 \end{itemize}
\end{lemma}
\begin{proof}
See \cite[Proposition 1.5.1]{Po04} and \cite[Theorem XIII.2.21]{Ta03b}.
\end{proof}

Note that if $c \in \rZ^{1}(\cR)$ and $\theta \in [\cR]$, then we have $\Ad(c)(u_\theta)=c(\theta)u_\theta$ where $c(\theta) \in \rL^{0}(X,\T)$ is defined by
\[ c(\theta)(x)=c(\theta^{-1}(x),x) \]
for a.e.\ $x \in X$.

Finally, we will later need the following lemma.

\begin{lemma}[{\cite[Theorem XIII.2.13]{Ta03b}}] \label{lemma: S}
Let $\cR$ be an ergodic equivalence relation on a standard measure space $X$. Then $\rS(\mathcal R) = \rS(\rL(\mathcal R))$.
\end{lemma}

%%%%%%%%%%%%%%%%%%%%%%%%%%%%%%%%%%%%%%%%%%%%%%

\section{Spectral gap for strongly ergodic equivalence relations}

We start this section by the following definition:

\begin{definition}
Let $(X,\mu)$ be a standard probability space. A nonsingular partial isomorphism $\theta$ of $(X,\mu)$ is said to be \emph{$\mu$-bounded} if the function $\log \left( \frac{\mathrm{d} (\mu \circ \theta)}{\mathrm{d}\mu} \right)$ is bounded on the support of $\theta$. In this case, the map $L_\theta : f \mapsto \theta(f)$ defines a bounded operator in $\B(\rL^{2}(X,\mu))$, where $\theta(f)=f \circ \theta^{-1}$.
\end{definition}

In the appendix, we show that such $\mu$-bounded partial isomorphisms are abundant in the full pseudo-group of any type $\III$ ergodic equivalence relation $\cR$ on a probability space $(X,\mu)$.

Our first main result in this section, Theorem \ref{general_gap} below, provides equivalent characterizations of the spectral gap property for finite sets of $\mu$-bounded automorphisms of $(X, \mu)$. 

\begin{theorem} \label{general_gap}
Let $(X,\mu)$ be a standard probability space. Let $\theta_1,\dots,\theta_n$ be a finite family of $\mu$-bounded automorphisms of $X$. Then the following assertions are equivalent: 
\begin{itemize}
\item [$(\rm i)$] There exists a constant $\kappa > 0$ such that for all measurable subsets $A \subset X$ we have
\[ \mu(A)(1-\mu(A)) \leq \kappa \sum_{k=1}^n \mu(A \triangle \theta_k(A)).\]
\item [$(\rm ii)$]  There exists a constant $\kappa > 0$ such that for all $f \in \rL^{2}(X,\mu)$, we have
\[ \| f-\mu(f) \|_2 \leq \kappa \sum_{k = 1}^n \| \theta_k(f)-f\|_{2}.\]
\item [$(\rm ii')$]  There exists a constant $\kappa > 0$ such that for all $f \in \rL^{2}(X,\mu)$, we have
\[ \| f-\mu(f) \|_2^{2} \leq \kappa \sum_{k = 1}^n \| \theta_k(f)-f\|_{2}^{2}.\]
\item [$(\rm iii)$]  The eigenvalue $0$ is a simple isolated point in the spectrum of
\[ T=\sum_{k = 1}^n |L_{\theta_k} - \mathbf{1} |  \in \B(\rL^{2}(X,\mu) ).\]
\item [$(\rm iii')$]  The eigenvalue $0$ is a simple isolated point in the spectrum of
\[ T'=\sum_{k = 1}^n |L_{\theta_k} - \mathbf{1} |^{2}  \in \B(\rL^{2}(X,\mu) ).\]
\end{itemize}
When these conditions are satisfied, we say that the family $\theta_1,\dots,\theta_n$ has a spectral gap.
\end{theorem}

The proof of Theorem \ref{general_gap} relies on the following lemma which is inspired by the trick of Namioka that was used in \cite{Co75b, Sc80}. Our new input is item $({\rm iii})$ which is crucial in the proof of Theorem \ref{general_gap}.

\begin{lemma} \label{spectral_resolution}
Let $(X,\mu)$ be a probability space. For $a \geq 0$ and $f$ a positive measurable function on $X$, we use the notation $e_a(f):=\mathbf{1}_{[a,+\infty)}(f)$. 
\begin{itemize}
\item [$(\rm i)$]  For every positive function $f \in \rL^1(X,\mu)$, we have
\[ \mu(f)=\int_0^\infty \mu(e_a(f)) \, \rd a.\]
\item [$(\rm ii)$]  For every positive function $f,g \in \rL^1(X,\mu)$, we have
\[  \|f-g\|_1=\int_0^\infty \mu(e_{a}(f) \triangle e_{a}(g)) \, \rd a.\]
\item [$(\rm iii)$]  For every positive function $f \in \rL^2(X,\mu)$ we have 
\[ \| f-\mu(f)  \|^2_2  \leq \int_0^\infty \mu( e_{a}(f^{2})) \left(1-\mu(e_{a}(f^{2})) \right)  \rd a.\]
\end{itemize}
\end{lemma}
\begin{proof}
$(\rm i)$ By Fubini's theorem, we have
\[ \int_0^\infty \mu(e_a(f)) \, \rd a= \int_0^\infty \int_{X} e_a(f(x)) \, \rd\mu(x) \, \rd a=\int_X \int_0^\infty e_a(f(x)) \, \rd a \, \rd \mu(x)=\int_X f(x) \, \rd \mu(x). \]

$(\rm ii)$ By Fubini's theorem, we have
\[ \int_0^\infty \mu(e_{a}(f) \triangle e_{a}(g)) \, \rd a =\int_X \int_0^\infty  | e_{a}(f(x)) - e_{a}(g(x)) | \, \rd a \, \rd \mu(x)=\int_X |f(x)-g(x)| \, \rd \mu(x). \]

$(\rm iii)$ First, note that $\| f- \mu(f) \|_2^2=\mu(f^2)-\mu(f)^2$ and that
\[ \mu(f^2)=\int_0^\infty \mu(e_a(f^2)) \, \rd a.\]
Therefore, we only have to show that 
\[ \int_0^\infty \mu( e_{a}(f^{2}))^2 \, \rd a  \leq \mu(f)^2. \]
On $(X , \mu) \otimes (X, \mu)$, we have $e_{a}(f^{2}) \otimes e_{a}(f^{2}) \leq e_{a}(f \otimes f)$. Hence, by appying $\mu \otimes \mu$ we get 
\[ \mu( e_{a}(f^{2}))^2 \leq (\mu \otimes \mu)(e_a(f \otimes f)), \]
and thus, after integrating over $a$ and using $(\rm i)$, we finally  get
\begin{equation*}
 \int_0^\infty \mu( e_{a}(f^{2}))^2 \, \rd a  \leq \int_0^\infty (\mu \otimes \mu)(e_a(f \otimes f)) \, \rd a \leq (\mu \otimes \mu)(f \otimes f)=\mu(f)^2. \qedhere
\end{equation*}
\end{proof}

\begin{proof}[Proof of Theorem \ref{general_gap}]
The equivalences $(\rm ii) \Leftrightarrow (\rm ii') \Leftrightarrow (\rm iii) \Leftrightarrow (\rm iii')$ are clear. On the other hand, $(\rm ii')$ applied to the indicator function $\mathbf{1}_A$ gives $(\rm i)$. Hence we only have to prove that $(\rm i) \Rightarrow (\rm ii)$. Assume that $(\rm i)$ is satisfied for some constant $\kappa > 0$.  Take a positive function $f \in \rL^{2}(X,\mu)$. Lemma \ref{spectral_resolution} gives
\[ \| f-\mu(f) \|_2^{2} \leq \kappa \sum_{k = 1}^n  \|f^{2}-\theta_k(f^{2}) \|_1. \]
Let $C=\max_k \| \mathbf{1} + L_{\theta_k} \|$. Then for all $k$, we have
\[ \|f^{2}-\theta_k(f^{2}) \|_1 \leq \|f+\theta_k(f) \|_2 \cdot \| f- \theta_k(f) \|_2 \leq  C \| f \|_2 \, \|f-\theta_k(f) \|_2, \]
hence we obtain 
\[ \| f-\mu(f) \|_2^{2} \leq \kappa C  \| f \|_2 \sum_{k = 1}^n \|f-\theta_k(f) \|_2.  \]

Now let $f \in \rL^2(X,\mu)$ be a real valued function with $\mu(f)=0$. Write $f=f_+-f_-$ where $f_+$ and $f_-$ are the positive and the negative part of $f$. Then we have
\[ \| f \|_2^2 \leq 2\left( \| f_+-\mu(f_+) \|_2^2+\| f_--\mu(f_-) \|_2^2 \right), \]
hence 
\[ \| f \|_2\leq 2 \kappa C  \sum_{k = 1}^n  \left( \|f_+-\theta_k(f_+) \|_2 + \|f_- -\theta_k(f_-) \|_2 \right),  \]
and since $\|f_{\pm}-\theta_k(f_{\pm}) \|_2 \leq \|f-\theta_k(f) \|_2 $, we obtain  
\[ \| f \|_2  \leq 4\kappa C  \sum_{k = 1}^n \|f-\theta_k(f) \|_2 .\]
Finally, by applying this inequality to $f-\mu(f)$ we obtain $(2)$ for all real valued functions $f \in \rL^2(X,\mu)$. Finally, for complex valued functions, one can just decompose over the real and imaginary part to get the desired result.
\end{proof}

As we saw, the proof of Theorem \ref{general_gap} is very elementary. It moreover provides a short proof of the following well-known result (see the proof of \cite[Proposition 2.3]{Sc80}).

\begin{corollary}
Let $\Gamma$ be any countable discrete group and $\Gamma \curvearrowright (X, \mu)$ any strongly ergodic pmp action. The following conditions are equivalent:
\begin{itemize}
\item [$(\rm i)$] The action $\Gamma \curvearrowright (X, \mu)$ does not have spectral gap.

\item [$(\rm ii)$] There exists an $I$-sequence in the sense of \cite[Section 2]{Sc80}, i.e.\ there exists a sequence $(A_i)_{i \in \N}$ of proper measurable subsets of $X$ such that $\mu(A_i) \to 0$ and $\frac{\mu(A_i \triangle \gamma A_i)}{\mu(A_i)} \to 0$ for all $\gamma \in \Gamma$.
\end{itemize}
\end{corollary}

\begin{proof}
We only have to prove $(\rm i) \Rightarrow (\rm ii)$. Since the action $\Gamma \curvearrowright (X, \mu)$ does not have spectral gap, the equivalence $(\rm i) \Leftrightarrow (\rm ii)$ in Theorem \ref{general_gap} implies the existence of a sequence $(A_i)_{i \in \N}$ of proper measurable subsets of $X$ such that 
\[ \frac{\mu(A_i \triangle \gamma(A_i))}{\mu(A_i)(1 - \mu(A_i))}  \rightarrow 0\]
for all $\gamma \in \Gamma$. Since the action $\Gamma \curvearrowright (X, \mu)$ is strongly ergodic, we infer that $\mu(A_i)(1 - \mu(A_i)) \to~0$. Up to replacing $A_i$ by $X \setminus A_i$ for each $i \in \N$, we may assume that $\mu(A_i) \to 0$ and hence $(A_i)_{i \in \N}$ is an $I$-sequence. 
\end{proof}

The second main result in this section, Theorem \ref{gap_strongly_ergodic} below, gives a spectral gap characterization of strongly ergodic equivalence relations.

\begin{theorem}[Theorem \ref{thmstar:spectral-gap}]\label{gap_strongly_ergodic}
Let $\mathcal{R}$ be a strongly ergodic equivalence relation on a standard probability space $(X,\mu)$ that either preserves $\mu$ or is of type $\III$. Then there exists a finite family of $\mu$-bounded elements $\theta_1,\dots,\theta_n \in [\mathcal{R}]$ that has spectral gap.
\end{theorem}

Before proving Theorem \ref{gap_strongly_ergodic}, we need to introduce some terminology.

\begin{definition}\label{small_invariant} 
Let $\mathcal R$ be an ergodic equivalence relation on a probability space $(X,\mu)$. We say that $\mathcal{R}$ admits \emph{small almost invariant sets} (with respect to $\mu$) if for any $\varepsilon > 0$ and any finite family of $\mu$-bounded elements $\theta_1,\dots,\theta_n \in [\mathcal{R}]$, we can find a nonzero measurable subset $A \subset X$ such that $\mu(A) < \varepsilon$ and 
\[  \forall k \in \{1,\dots,n \}, \quad \mu(A \triangle \theta_k(A)) \leq \varepsilon \mu(A).\]
\end{definition}

The following lemma will be a crucial step in the proof of Theorem \ref{gap_strongly_ergodic}. 

\begin{lemma} \label{corner}
Let $\mathcal{R}$ be an ergodic equivalence relation on a probability space $(X,\mu)$ that either preserves $\mu$ or is of type $\III_\lambda, \; 0 < \lambda \leq 1$. Suppose that $\mathcal{R}$ admits small almost invariant sets. Then for any nonzero measurable subset $Y \subset X$ and any finite family of $\mu$-bounded elements $\theta_1,\dots,\theta_n \in [\mathcal{R}]$, we can find a nonzero measurable subset $A \subset Y$ such that $\mu(A) < \varepsilon$ and 
\[  \forall k \in \{1,\dots,n \}, \quad \mu(A \cap \theta_k(Y \setminus A)) + \mu((Y \setminus A) \cap \theta_k(A)) \leq \varepsilon \mu(A). \]
In particular, the reduced equivalence relation $\mathcal{R}_{Y}$ admits small almost invariant sets.
\end{lemma}

\begin{proof}
Since $\mathcal{R}$ admits small almost invariant sets, we can find a net $(A_i)_{i \in I}$ of measurable subsets of $X$ with $\mu(A_i) > 0$ for all $i$, such that $\mu(A_i) \to 0$ and
\[ \frac{\mu(A_i \triangle \theta(A_i))}{\mu(A_i)}  \rightarrow 0 \]
 when $i \to \infty$ for all $\mu$-bounded $\theta \in [\mathcal{R}]$. Up to extracting a subnet, we may also assume that the net of probability measures $\mu_i:=\frac{1}{\mu(A_i)} \mu(A_i \cap \, \cdot \, )$ converges in the weak$^\ast$ topology to some positive linear functional $\varphi$ in the unit ball of $\rL^\infty(X,\mu)^*$. Now, let $B_i = A_i \cap Y$. We want to show that $\mu(B_i) > 0$ for $i$ large enough and that
\[ \frac{1}{\mu(B_i)} \left( \mu(B_i \cap \theta(Y \setminus B_i)) + \mu((Y \setminus B_i) \cap \theta(B_i)) \right) \to 0 \]
as $i \to \infty$ for all $\mu$-bounded $\theta \in [\mathcal{R}]$.

In order to show this, note that for every $\theta \in [\mathcal{R}]$ we have
\[  \mu(B_i \cap \theta(Y \setminus B_i)) + \mu((Y \setminus B_i) \cap \theta(B_i)) \leq \mu(A_i \triangle \theta(A_i)).\]
Hence, it will be enough to show that $\frac{\mu(B_i)}{\mu(A_i)}$ converges to some positive number, i.e.\ that $\varphi(\mathbf{1}_Y) >~0$. 

If $\mathcal{R}$ preserves $\mu$, we know by construction that $\varphi$ is invariant by $[\mathcal{R}]$. Hence we have $\varphi(\mathbf 1_{\theta(Y)})=\varphi(\mathbf 1_Y)$ for all $\theta \in [\mathcal{R}]$. Since we can find finitely many $\theta_1, \dots, \theta_n \in [\mathcal{R}]$ such that $\mathbf 1_X \leq \sum_{k=1}^n \mathbf 1_{\theta_n(Y)}$, we conclude that $1=\varphi(\mathbf 1_X) \leq n \, \varphi(\mathbf 1_Y)$. This shows that $\varphi(\mathbf 1_Y) > 0$.

If $\mathcal{R}$ is of type $\III_\lambda$, $0 < \lambda \leq 1$, we can construct, by Theorem \ref{bounded_typeIII}, a $\mu$-bounded element $\theta \in [\mathcal{R}]$ such that $\theta(X \setminus Y) \subset Y$. Since $\theta$ is $\mu$-bounded, we can find $C > 0$ such that $\mu  \leq C (\mu \circ \theta)$. Then, since the net $(A_i)_{i \in I}$ is almost invariant by $\theta$, we can see that $\varphi  \leq C (\varphi \circ \theta)$. Hence $\varphi(\mathbf 1_{X \setminus Y}) \leq C \, \varphi(\mathbf 1_Y)$. This again shows that $\varphi(\mathbf 1_Y) > 0$. 
\end{proof}

\begin{proof}[Proof of Theorem \ref{gap_strongly_ergodic}]
We suppose that such a family does not exist and we contradict the strong ergodicity of $\mathcal{R}$. Using the negation of Theorem \ref{general_gap} $(\rm i)$, there exists a net $(A_i)_{i \in I}$ of proper measurable subsets of $X$ such that 
\[ \frac{\mu(A_i \triangle \theta(A_i))}{\mu(A_i)(1-\mu(A_i))}  \rightarrow 0 \]
for all $\mu$-bounded elements $\theta \in [\mathcal{R}]$. Since the set of $\mu$-bounded elements is dense in $[\mathcal R]$ by Corollary \ref{cor-dense}, we infer that $\mu(A_i \triangle \theta(A_i)) \to 0$ for all $\theta \in [\mathcal{R}]$. Hence, by the strong ergodicity of $\mathcal{R}$, we know that $\mu(A_i)(1-\mu(A_i)) \to 0$. Then, by replacing the set $A_i$ by $X \setminus A_i$ for each $i$ if necessary, we can assume that $\mu(A_i) \to 0$. Hence $\mathcal{R}$ admits small almost invariant sets. Now, fix $\theta_1,\dots, \theta_n$ a finite set of $\mu$-bounded elements in $[\mathcal{R}]$ and take $\varepsilon > 0$. Consider the set $\Lambda$ of all elements $A \in \mathfrak{P}(X)$ such that $\mu(A) \leq \frac{1}{2}$ and
\[ \forall k \in \{1,\dots,n \}, \quad \mu(A \triangle \theta_k(A)) \leq \varepsilon \mu(A). \]
Since $\Lambda$ is closed in the complete boolean algebra $\mathfrak{P}(X)$, it is inductive as a poset. Let $A \in \Lambda$ be a maximal element. Suppose that $\mu(A) < \frac{1}{2}$. Let $Y=X \setminus A$. By Lemma \ref{corner}, we can find $B \subset Y$ such that $0 < \mu(B) < \frac{1}{2}- \mu(A)$ and
\[ \forall k \in \{1,\dots,n \}, \quad \mu(B \cap \theta_k(Y \setminus B))+\mu((Y \setminus B) \cap \theta_k(B))  \leq \varepsilon \mu(B). \]
Then for $C=A \cup B$, we easily check that
\[ \forall k \in \{1,\dots,n \}, \quad \mu(C \cap \theta_k(X \setminus C))+\mu((X \setminus C) \cap \theta_k(C))  \leq \varepsilon \mu(C) \]
and $\mu(C)=\mu(A)+\mu(B) < \frac{1}{2}$. Therefore $C \in \Lambda$ and this contradicts the maximality of $A$. Hence we must have $\mu(A)=\frac{1}{2}$. We have proved that for any finite family of $\mu$-bounded elements in $[\mathcal{R}]$ and any $\varepsilon > 0$, we can find $A \subset X$ such that $\mu(A)=\frac{1}{2}$ and
\[ \forall k \in \{1,\dots,n \}, \quad \mu(A \triangle \theta_k(A)) \leq \varepsilon. \]
This contradicts the strong ergodicity of $\mathcal{R}$.
\end{proof}

\begin{remark}
Let us point out that in the case when $\mathcal R$ is pmp and ergodic, the existence of small almost invariant sets for $\mathcal R$ does not guarantee the existence of an $I$-sequence for the full group $[\mathcal R]$ (in the sense of \cite[Section 2]{Sc80}), but rather an $I$-{\em net}. This is because the full group $[\mathcal R]$ is uncountable. For that reason, we cannot use \cite[Proposition 2.3]{Sc80} to prove Theorem \ref{gap_strongly_ergodic} even in the case when $\mathcal R$ is pmp. 
\end{remark}

\section{Strong ergodicity of skew-product equivalence relations}

Let $G$ be a second countable locally compact abelian group. Let $\cR$ be a nonsingular equivalence relation on a standard measure space $X$ and $\Omega  \in \rZ^{1}(\cR,G)$ a measurable $1$-cocycle. Our goal in this section is to give a criterion for the strong ergodicity of the skew-product equivalence relation $\cR \times_\Omega G$. For this, we let $\widehat{G}$ to be the Pontryagin dual of $G$ and we introduce the map $\widehat{\Omega} : \widehat{G} \rightarrow \rZ^{1}(\cR)$ defined by $\widehat{\Omega}(p)(x,y)=\langle p , \Omega(x,y) \rangle$ for a.e.\ $(x,y) \in \cR$, where $\langle \, \cdot \, , \cdot \, \rangle : \widehat{G} \times G \rightarrow \T$ is the duality pairing. Note that $\widehat{\Omega}$ is a continuous group homomorphism. We also introduce the continuous homomorphism $[ \widehat{\Omega}] : \widehat{G} \rightarrow \rH^{1}(\cR)$ which sends $p \in \widehat{G}$ to the cohomology class $[ \widehat{\Omega}(p)] \in \rH^{1}(\cR)$.

The main theorem of this section is the following result:
\begin{theorem}[Theorem \ref{thmstar:skew-product}] \label{strongly_ergodic_ext}
Let $\cR$ be an ergodic equivalence relation on a standard measure space $(X, \mu)$. Let $\Omega \in \rZ^1(\cR,G)$ be any measurable $1$-cocycle with values into a locally compact second countable abelian group $G$. Consider the following assertions:
\begin{itemize}
\item [$(\rm i)$] The skew-product equivalence relation $\cR \times_\Omega G$ is strongly ergodic.
\item [$(\rm ii)$] The equivalence relation $\cR$ is strongly ergodic and the map $[\widehat{\Omega}] : \widehat{G} \rightarrow \rH^1(\cR)$ is a homeomorphism onto its range.
\end{itemize}
Then $(\rm i) \Rightarrow (\rm ii)$ and if $G$ contains a lattice, we also have $(\rm ii) \Rightarrow (\rm i)$.
\end{theorem}

We first prove the following proposition which deals with the ergodicity of the skew-product equivalence relation. It gives a hint for the proof of Theorem \ref{strongly_ergodic_ext}. Note that this proposition is closely related to \cite[Corollary 3.8]{Zi76}. Indeed, if $\Omega$ is cohomologous to some other $1$-cocycle $\Omega' \in \rZ^1(\cR,G)$ such that $\overline{\mathrm{Range}}(\Omega')$ is contained in some closed subgroup $H \subset G$ then $H^\perp \subset \ker ([\widehat{\Omega}])$.

\begin{proposition} \label{ergodic_ext}
Let $\cR$ be a nonsingular equivalence relation on a standard measure space $X$. Let $\Omega \in \rZ^1(\cR,G)$ be a measurable $1$-cocycle with values into a second countable locally compact abelian group $G$. Consider the following assertions:
\begin{itemize}
\item [$(\rm i)$] The skew-product equivalence relation $\cR \times_\Omega G$ is ergodic.
\item [$(\rm ii)$] The relation $\cR$ is ergodic and the map $[ \widehat{\Omega}] : \widehat{G} \rightarrow \rH^1(\cR)$ is injective.
\end{itemize}
Then $(\rm i) \Rightarrow (\rm ii)$ and if $G$ is compact, we also have $(\rm ii) \Rightarrow (\rm i)$.
\end{proposition}
\begin{proof}
$(\rm i) \Rightarrow (\rm ii)$. Suppose that $\cR \times_\Omega G$ is ergodic. Since $\cR \times_\Omega G$ clearly surjects onto $\cR$, we know that $\cR$ is also ergodic. Moreover, this surjection induces a natural inclusion $\iota : \rZ^1(\cR) \rightarrow \rZ^1(\cR \times_\Omega G)$. Let $p \in \widehat{G}$ and suppose that $\widehat{\Omega}(p)=\partial u$ for some $u \in \rL^0(X,\T)$. By definition of the skew-product construction, we have the following relation $\partial(1 \otimes p)=\iota(\widehat{\Omega}(p))$ where we view $1 \otimes p$ as an element of $\rL^0(X \times G,\T)$. Therefore we have $\partial(1 \otimes p)=\partial(u \otimes 1)$. Since $\cR \times_{\Omega} G$ is ergodic, this means that $1 \otimes p=\lambda(u \otimes 1)$ for some $\lambda \in \T$ and this easily implies that $p=1$. Hence $[ \widehat{\Omega}]$ is injective.

$(\rm ii) \Rightarrow (\rm i)$ when $G$ is compact. Let $\Gamma=\widehat{G}$ be the dual discrete group.  Let $f \in \rL^\infty(X \times G)^{\cR \times_\Omega G}$. Consider its Fourier decomposition given by the formal sum
\[ f= \sum_{ \gamma \in \Gamma} f_\gamma \otimes \gamma, \]
where $f_\gamma \in \rL^\infty(X)$ for all $\gamma \in \Gamma$. Then, since $f$ is invariant, we have
\[ 0=f^{\ell}-f^{r}=\sum_{\gamma \in \Gamma} \left( \widehat{\Omega} (\gamma) f_\gamma^{\ell} - f_\gamma^{r} \right) \otimes \gamma^{r}, \]
where we used the relation $(1 \otimes \gamma)^{\ell}=\iota(\widehat{\Omega}(\gamma)) (1 \otimes \gamma)^{r}$.

This shows that $f_\gamma^{r}= \widehat{\Omega} (\gamma) f_\gamma^{\ell}$ for all $\gamma \in \Gamma$. In particular, we have $|f_\gamma |^{\ell}=|f_\gamma |^{r}$, and since $\cR$ is ergodic, this means that $|f_\gamma |$ is a constant for all $\gamma \in \Gamma$. Suppose that for some $\gamma$, we have $f_\gamma \neq 0$. Let $u_\gamma$ be the polar part of $f_\gamma$. Then we have $\widehat{\Omega}(\gamma)=\overline{\partial u_\gamma}$. Since $[ \widehat{\Omega}]$ is injective, this implies that $\gamma=e$. We then have $f = f_e \otimes e$ and $\partial u_e=1$. Since $\cR$ is ergodic, this means that $u_e \in \T$ and therefore $f_e$ is constant. We conclude that $f$ is constant. Hence $\cR \times_\Omega G$ is ergodic. 
\end{proof}

For the proof of Theorem \ref{strongly_ergodic_ext}, we will need the following proposition. We prove it by using the ultraproduct construction.

\begin{proposition} \label{amenable}
Let $\cR$ be an ergodic equivalence relation on a standard measure space $X$. Let $\Gamma \curvearrowright X$ be a nonsingular action of an amenable countable discrete group by automorphisms of $\cR$. Then $\cR$ is strongly ergodic if and only if the equivalence relation generated by $\cR$ and $\cR(\Gamma \curvearrowright X)$ is strongly ergodic.
\end{proposition}

\begin{proof}
Clearly, if $\cR$ is strongly ergodic then a fortiori the equivalence relation generated by $\cR$ and $\cR(\Gamma \curvearrowright X)$ is strongly ergodic. Now, suppose that $\cR$ is not strongly ergodic. Choose $\mu$ a probability measure in the measure class of $X$ and put $\tau(f) = \int_{X} f \, \rd \mu$ for every $f \in \rL^\infty(X)$. Fix a nonprincipal ultrafilter $\omega \in \beta(\N) \setminus \N$ and consider the ultraproduct von Neumann algebra $\rL^\infty(X)^\omega$ together with its canonical faithful normal trace $\tau^\omega$ defined by $\tau^\omega((f_n)^\omega) = \lim_{n \to \omega} \tau(f_n)$ for every $(f_n)^\omega \in \rL^\infty(X)^\omega$. Define the (abelian) von Neumann subalgebra $\mathcal A \subset \rL^\infty(X)^\omega$ by 
$$\mathcal A:= \left\{ (f_n)^\omega \in \rL^\infty(X)^\omega : \lim_{n \to \omega} \|\theta(f_n) - f_n\|_2 = 0, \forall \theta \in [\cR] \right \}.$$
By assumption, we have $\mathcal A \neq \C1$. We claim that $\mathcal A$ is diffuse. Indeed, following the proof of \cite[Corollary 3.8]{Co74}, let $p \in \mathcal A$ be any projection such that $p \neq 0, 1$. Write $p = (p_n)^\omega$ and $\alpha = \tau^\omega(p) \in (0, 1)$. By \cite[Proposition 1.1.3 (a)]{Co75a}, we may further assume that $p_n \in \rL^\infty(X)$ is a projection for every $n \in \N$. Since $\cR$ is ergodic and since $\lim_{n \to \omega} \|\theta(p_n) - p_n\|_2 = 0$, it follows that $p_n \to \alpha 1$ weakly as $n \to \omega$. Fix a countable dense subset $\{\theta_n : n \in \N\} \subset [\cR]$. For every $n \in \N$, there exists $k_n \in \N$ large enough such that 
$$\forall 0 \leq j \leq n, \quad \|\theta_j(p_{k_n}) - p_{k_n}\|_2 \leq \frac{1}{n + 1} \quad \text{and} \quad |\tau(p_j p_{k_n}) - \alpha \tau(p_j)| \leq \frac{1}{n + 1}.$$
Then $q = (p_n p_{k_n})^\omega \in \mathcal A$, $q \leq p$ and $\tau^\omega(q) = \alpha^2$. This shows that $\mathcal A$ has no minimal projection and hence is diffuse. 

Observe that the natural action $\Gamma \curvearrowright \rL^\infty(X)^\omega$ defined by $\gamma \cdot (f_n)^\omega = (\gamma(f_n))^\omega$ for every $(f_n)^\omega \in \rL^\infty(X)^\omega$ and every $\gamma \in \Gamma$ leaves the von Neumann subalgebra $\mathcal A \subset \rL^\infty(X)^\omega$ globally invariant. Since $\mathcal A$ is diffuse and $\Gamma$ is countable, we may find a diffuse von Neumann subalgebra $\mathcal D \subset \mathcal A$ with separable predual that is globally invariant under the action $\Gamma \curvearrowright \mathcal A$. Write $\mathcal D = \rL^\infty(Y, \eta)$ where $(Y, \eta)$ is a diffuse standard probability space. Since $\Gamma$ is amenable and $(Y, \eta)$ is diffuse, the action $\Gamma \curvearrowright (Y, \eta)$ is not strongly ergodic (see e.g.\ \cite[Proposition 2.2]{Sc79}). Thus, there exists a $\Gamma$-almost invariant sequence of measurable subsets $U_k \subset Y$ such that $\inf_{k \in \N} \eta(U_k)(1 - \eta(U_k)) > 0$. For every $k \in \N$, write $\rL^\infty(Y, \eta) \ni \mathbf 1_{U_k}  = (p^k_n)^\omega \in \mathcal D \subset \mathcal A$ where $(p^k_n)_n$ is a sequence of projections in $\rL^\infty(X)$. By diagonal extraction, we may find a sequence of projections $(q_m)_m$ in $\rL^\infty(X)$ of the form $q_m = p^{k_m}_{n_m}$ such that $\inf_{m \in \N} \tau(q_m)(1 - \tau(q_m)) > 0$, $\lim_m\|\gamma(q_m) - q_m \|_2 = 0$ for every $\gamma \in \Gamma$ and $\lim_m \|\theta(q_m) - q_m\|_2 = 0$ for every $\theta \in [\cR]$. For every $m \in \N$, write $q_m = \mathbf 1_{V_m}$ where $V_m \subset X $ is a measurable subset. Then $(V_m)_m$ is a nontrivial sequence of measurable subsets that are almost invariant under both $\cR$ and $\cR(\Gamma \curvearrowright X)$. Therefore $\cR$ and $\cR(\Gamma \curvearrowright X)$ generate an equivalence relation that is not strongly ergodic.
\end{proof}

The next lemma will be the crucial ingredient in the proof of Theorem \ref{strongly_ergodic_ext}. It gives a spectral gap control for the topology of $\rH^{1}(\cR)$ when $\cR$ is a strongly ergodic equivalence relation. Of course, it relies on Theorem \ref{gap_strongly_ergodic}. 

If $c \in \rZ^{1}(\cR)$ and $\theta \in [\cR]$, we introduce the function $c(\theta) \in \rL^{0}(X,\T)$ defined by $c(\theta)(x)=c(\theta^{-1}(x),x)$ for a.e.\ $x \in X$. Note that a sequence $c_n \in \rZ^{1}(\cR)$ converges to $\mathbf 1$ in the measure topology if and only if $c_n(\theta)$ converges to $\mathbf 1$ in the measure topology for all $\theta \in [\cR]$.

\begin{lemma} \label{gap_cocycles}
Let $\cR$ be a strongly ergodic equivalence relation on a standard probability space $(X,\mu)$. Assume that $\cR$ either preserves $\mu$ or is of type $\III$. Let $\mathcal{V}$ be a neighborhood of the identity in $\rH^1(\cR)$. Then there exist a constant $\kappa > 0$ and a finite family of $\mu$-bounded elements $\theta_1,\dots,\theta_n \in [ \cR]$ such that 
\[ \forall f \in \rL^2(X,\mu), \; \|f -\mu(f)\|_2 \leq \kappa \sum_{k=1}^n \| f-\theta_k(f) \|_2 \] 
and such that for all cocycles $c \in \rZ^1(\cR)$ with $[c] \notin \mathcal{V}$, we have
\[ \forall f \in \rL^2(X,\mu), \; \|f \|_2 \leq \kappa \sum_{k=1}^n \| c(\theta_k)f-\theta_k(f) \|_2. \] 
\end{lemma}

\begin{proof}
Using Theorem \ref{gap_strongly_ergodic}, we may choose a finite family $\theta_1,\dots, \theta_n$ of $\mu$-bounded elements in $[\cR]$ with a constant $\kappa > 0$ such that for all $f \in \rL^2(X,\mu)$ we have
\[ \|f-\mu(f) \|_2 \leq \kappa \sum_{k=1}^n \| \theta_k(f)-f \|_2 \]
and for all $c \in \rZ^1(\cR)$ with $[c] \notin \mathcal{V}$ we have
\[ \frac{1}{\kappa} \leq \sum_{k=1}^n \| c(\theta_k)-\mathbf 1 \|_2. \]
Now, suppose by contradiction that there exists a sequence $(f_i)_{ i \in \N}$ in $\rL^2(X,\mu)$ with $\|f_i \|_2=1$ for all $i \in \N$  and a sequence of measurable $1$-cocycles $c_i \in \rZ^1(\cR)$ with $[c_i] \notin \mathcal{V}$ for all $i \in \N$  such that $\lim_i \| c(\theta_k)f_i-\theta_k(f_i) \|_2 =0$ for all $k\in \{1,\dots, n\}$. Then in particular, we have $\lim_i \| |f_i|-\theta_k(|f_i|) \|_2=0$ for all $k\in \{1,\dots, n\}$. Therefore, we have $\lim_i \| |f_i |- \mu(|f_i|) \|_2=0$. Since $\|f_i \|_2=1$ for all $i \in \N$, this means that $\lim_i \| |f_i|-\mathbf 1 \|_2=0$. Thus, we can find a sequence $u_i \in \rL^0(X,\T)$ such that $\lim_i \| f_i-u_i \|_2=0$. Since $\theta_k$ is $\mu$-bounded for all $k\in \{1,\dots, n\}$, we then obtain
\[ \lim_i \| c_i(\theta_k)u_i-\theta_k(u_i) \|_2=0 \]
for all $k\in \{1,\dots, n\}$. Let $\partial u_i \in \rB^1(\cR)$ be the measurable $1$-coboundary associated with $u_i$ and let $c'_i= (\partial u_i)^{-1}c_i \in \rZ^1(\cR)$. Then we have
\[ \lim_i \| c'_i(\theta_k)- \mathbf 1\|_2=0. \]
But, since $[c'_i]=[c_i] \notin \mathcal{V}$, we also have
\[ \sum_{k=1}^n \| c'_i(\theta_k)- \mathbf 1 \|_2 \geq \frac{1}{\kappa} \]
for all $i \in \N$ and this is a contradiction.
\end{proof}

\begin{proof}[Proof of Theorem \ref{strongly_ergodic_ext}]
$(\rm i) \Rightarrow (\rm ii)$. Suppose that $\cR \times_\Omega G$ is strongly ergodic. Since $\cR \times_\Omega G$ clearly surjects onto $\cR$, we know that $\cR$ is also strongly ergodic. Moreover, we have a natural embedding of topological groups $\iota : \rZ^1(\cR) \rightarrow \rZ^1(\cR \times_\Omega G)$.  Now we have to show that if we have a sequence $(p_i)_{i \in \N}$ of elements in $\widehat{G}$ and a sequence $(u_i)_{i \in \N}$ of elements in $\rL^0(X,\T)$ such that $\lim_i (\partial u_i)^{-1} \widehat{\Omega}(p_i)=\mathbf{1}$ in $\rZ^1(\cR)$ then $\lim_i p_i =1$ in $\widehat{G}$. By definition of the skew-product construction, it is easy to check that $\iota((\partial u_i)^{-1} \widehat{\Omega}(p_i))=\partial( u_i^{-1} \otimes p_i)$ where we view $u_i^{-1} \otimes p_i$ as an element of $\rL^0(X \times G, \T)$. Since $\iota$ is an embedding of topological groups, we have $\lim_i \partial( u_i^{-1} \otimes p_i)=\mathbf{1}$ in $\rZ^1(\cR \times_\Omega G)$. Since $\cR \times_\Omega G$ is strongly ergodic, this means that there exists a sequence $z_i \in \T$ such that $\lim_i z_i(u_i^{-1} \otimes p_i)=\mathbf{1}$ and this easily implies that $\lim_i p_i=1$ in $\widehat{G}$.

Now, we prove $(\rm ii) \Rightarrow (\rm i)$. First we deal with the case where $G$ is compact because we will need it for the more general case where $G$ contains a lattice.

$(\rm ii) \Rightarrow (\rm i)$ when $G$ is compact. Assume that $(\rm ii)$ holds. Then, thanks to Corollary \ref{corner_homology}, we know that $(\rm ii)$ also holds for the relation $\cR_Y$ and the cocycle $\Omega |_{ \cR_Y}$ where $Y \subset X$ is any nonzero measurable subset. Moreover, by Proposition \ref{ergodic_ext}, we know that $\cR \times_\Omega G$ is ergodic which means that $\cR \times_\Omega G$ is strongly ergodic if and only if $(\cR \times_\Omega G)|_{ Y \times G}$ is strongly ergodic (again by Corollary \ref{corner_homology}). This shows that it is enough to prove the desired result for $\cR_{Y}$ instead of $\cR$. In particular, we can assume that $\cR$ is either of type $\II_1$ (in which case we choose $\mu$ to be the unique invariant probability measure) or of type $\III$. Now, let $\Gamma =\widehat{G}$ be the dual discrete group.  Then for every $f \in \rL^2(X,\mu) \otimes \rL^2(G,\nu)$ we have
\[ f^{\ell}-f^r= \sum_{\gamma \in \Gamma } \left( \widehat{\Omega} (\gamma)  f_\gamma^{\ell}-f_\gamma^{r} \right) \otimes \gamma^{r}. \]
Therefore, by restricting to the graph of an element $\theta \in [ \cR] \subset [\cR \times_\Omega G]$ we obtain
\[ \| \theta(f)-f \|_2^2 = \sum_{ \gamma \in \Gamma } \| \theta(f_\gamma)- \overline{\widehat{\Omega}(\gamma)(\theta)} f_\gamma \|_2^2. \]
By assumption, we can take $\mathcal{V}$ to be a neighborhood of the identity in $\rH^1(\cR)$ such that $\pi_\cR^{-1}(\mathcal{V}) \cap \widehat{\Omega}(\Gamma)=\{1\}$. Take $\theta_1, \dots, \theta_n \in [\cR]$ and $\kappa > 0$ as in Lemma \ref{gap_cocycles}. Since we have
\[ \sum_{k=1}^n \| \theta_k(f)-f \|_2^2=\sum_{k=1}^n \| \theta_k(f_e)-f_e \|_2^2+ \sum_{ \gamma \neq e } \sum_{k=1}^n \| \theta_k(f_\gamma)- \overline{\widehat{\Omega}(\gamma)(\theta_k)}f_\gamma \|_2^2, \]
we obtain
\[ \kappa \sum_{k=1}^n \| \theta_k(f)-f \|_2^2 \geq  \|f_e-\mu(f_e) \|_2^2 + \sum_{ \gamma \neq e } \|f_\gamma \|_2^2 =\|f- (\mu \otimes \nu)(f) \|_2^2. \]
This shows that $\theta_1,\dots,\theta_k$ has spectral gap in $\cR \times_\Omega G$ and in particular that $\cR \times_\Omega G$ is strongly ergodic.

$(\rm ii) \Rightarrow (\rm i)$ when $G$ contains a lattice. By assumption, $G$ contains a discrete subgroup $H$ such that the quotient group $K=G/H$ is compact. Let $\Theta \in \rZ^1(\cR,K)$ be the measurable $1$-cocycle obtained by composing $\Omega$ with the quotient map $G \rightarrow K$. Then $\widehat{\Theta}=\widehat{\Omega}|_{\widehat{K}}$ so that $\Theta$ also satisfies the assumption $(\rm ii)$. Since $K$ is compact, we can apply the first step and therefore we know that $\cR \times_{\Theta} K$ is strongly ergodic. Now, let $q : X \times G \rightarrow X \times K$ be the quotient map and consider the lifted equivalence relation $\mathcal{Q}=q^{*}(\cR \times_{\Theta} K)$. Then $\mathcal{Q}$ is strongly ergodic because it is isomorphic to the product equivalence relation $(\cR \times_{\Theta} K) \times \mathcal{S}_H$ on $(X \times K) \times H$ where $\mathcal{S}_H$ is the type $\I$ transitive equivalence relation on $H$. Moreover, $\mathcal{Q}$ is generated by $\cR \times_\Omega G$ and the orbit equivalence relation $\cR(H \curvearrowright X \times G)$ of the translation action $H \curvearrowright X \times G$ on the second coordinate. By Proposition \ref{amenable}, we conclude that $\cR \times_\Omega G$ is strongly ergodic.
\end{proof}

A direct application of Theorem \ref{strongly_ergodic_ext} gives the following characterization of the strong ergodicity of the Maharam extension.

\begin{corollary}[Corollary \ref{corstar:maharam}]
Let $\cR$ be a type $\III_1$ ergodic equivalence relation on a standard measure space $X$. Then the Maharam extension $\core(\cR)$ is strongly ergodic if and only if $\cR$ is strongly ergodic and $\tau(\cR)$ is the usual topology on $\R$.
\end{corollary}

\section{Almost periodic equivalence relations}
Let $\mathcal R$ be an ergodic equivalence relation on a standard measure space $X$. We say that a measure $\mu \in \mathcal{M}(X)$ is {\em almost periodic} for $\mathcal R$ if $\delta_\mu$ is a step function. We say that $\mu$ is $\Lambda$-almost periodic for a countable subgroup $\Lambda < \R^+_0$ if moreover $\mathrm{Range}(\delta_\mu) \subset \Lambda$. It is easy to check that a measure $\mu$ is almost periodic if and only if the image of the homomorphism 
$$\widehat{\delta_\mu} : \R \to \rZ^{1}(\cR) : t \mapsto \delta_\mu^{\mathbf{i}t }$$ has a compact closure. It is $\Lambda$-almost periodic if and only if $\widehat{\delta_\mu}$ extends to a continuous homomorphism from the compact group $\widehat{\Lambda}$ to $\rZ^{1}(\cR)$.

We denote by $\mathcal M_{\ap}(X, \mathcal R)$ the set of all almost periodic measures on $X$. We say that $\cR$ is almost periodic if $\mathcal M_{\ap}(X, \mathcal R) \neq \emptyset$.

By analogy with \cite{Co74}, we introduce the Sd invariant for ergodic equivalence relations. It is a discrete version of the S invariant.

\begin{definition}
Let  $\mathcal R$ be an almost periodic ergodic equivalence relation on $X$. Define 
$$\Sd(\mathcal R) = \bigcap_{\mu \in \mathcal M_{\ap}(X, \mathcal R)} \Range(\delta_\mu).$$
\end{definition}

Observe that when $\mathcal R$ is ergodic and of type ${\rm I}$ or ${\rm II}$, then $\mathcal R$ is almost periodic and $\Sd(\mathcal R) = \mathbf 1$. More generally, we show that $\Sd(\mathcal R) < \R^+_0$ is a countable subgroup.

\begin{proposition}\label{proposition: group}
Let  $\mathcal R$ be an almost periodic ergodic equivalence relation on $X$. Then $\Sd(\mathcal R) < \R^+_0$ is a countable subgroup. 
\end{proposition}

\begin{proof}
We may assume that $\mathcal R$ is of type ${\rm III}$. It is easy to see that $\Sd(\mathcal R)$ is stable under taking inverses. It remains to show that $\Sd(\mathcal R)$ is stable under taking products. Let $\lambda_1, \lambda_2 \in \Sd(\mathcal R)$. Fix a measure $\mu \in \mathcal M_{\ap}(X, \mathcal R)$. Since $\lambda_1 \in \Sd(\mathcal R)$, we have $\lambda_1 \in \Range(\delta_\mu)$. Then there exists a nonzero $\theta \in [[\cR]]$ such that $\delta_\mu(\theta(x),x)=\lambda_1$ for a.e.\ $x \in \dom(\theta)$. Since $\mathcal R$ is of type ${\rm III}$, there exists $\phi \in [[\cR]]$ such that $\dom(\phi)=\ran(\theta)$ and $\ran(\phi)=X$. Then $\nu = \phi_*\mu|_{\dom(\phi)} \in \mathcal M_{\ap}(X, \mathcal R)$. Since $\lambda_2 \in \Sd(\mathcal R)$, we have $\lambda_2 \in \Range(\delta_\nu)$. Then there exists $\psi \in [[\cR]]$ such that $\delta_\nu(\psi(y), y) = \lambda_2$ for a.e.\ $y \in \dom(\psi)$. Then $\varphi = \phi^{-1}\psi \phi \theta \in [[\cR]]$ is nonzero and satisfies
\begin{align*}
\delta_\mu(\varphi(x),x) &= \delta_\mu( \phi^{-1}\psi \phi\theta(x),\theta(x)) \, \delta_\mu(\theta(x),x)   \\
&= \delta_\nu(\psi \phi\theta(x),\phi\theta(x)) \, \delta_\mu(\theta(x),x)  \\
&= \lambda_2 \lambda_1.
\end{align*}
for a.e.\ $x \in \dom(\theta)$. This shows that $\lambda_2 \lambda_1 \in \Range(\delta_\mu)$. Since this holds true for every $\mu \in \mathcal M_{\ap}(X, \cR)$, we obtain that $\lambda_2 \lambda_1 \in \Sd(\mathcal R)$.
\end{proof}

We obtain an analogue of \cite[Theorem 4.1]{Co74}.

\begin{theorem}\label{theorem: Sd invariant}
Let $\mathcal R$ be an almost periodic  ergodic equivalence relation on $X$. Assume that $\cR$ is strongly ergodic. Then, for any $\mu \in \mathcal{M}_{\ap}(X,\cR)$, we have
$$\Sd(\mathcal R) = \bigcap_{U  \in \mathfrak{P}(X), \; U \neq \emptyset} \Range(\delta_{\mu}|_{\cR_U}).$$
where $\cR_U=\cR \cap ( U \times U)$ is the restricted equivalence relation.
\end{theorem}

Before proving Theorem \ref{theorem: Sd invariant}, we need some preparation.  

\begin{lemma}\label{lemma: cohomology}
 Let $\Lambda < \R^+_0$ be any countable subgroup. Let $\nu_1, \nu_2 \in \mathcal M_{\ap}(X, \cR)$ be any $\Lambda$-almost periodic measures. Then there exists $h \in \rL^{0}(X,\Lambda)$ such that $$\delta_{\nu_2} \delta_{\nu_1}^{-1} = \partial h \in  \rB^1(\mathcal R, \Lambda).$$
\end{lemma}

\begin{proof}
Without loss of generality, we may assume that $\Lambda \neq \mathbf 1$. Regarding $\Lambda$ as a discrete group, we denote by $G = \widehat \Gamma$ the Pontryagin dual of $\Gamma$. Then $G$ is a second countable compact abelian group and we regard $\R = \widehat{\R^+_0} < \widehat \Gamma =  G$ as a dense subgroup. Let $\nu \in \mathcal M_{\ap}(X, \cR)$ be any $\Lambda$-almost periodic measure. Then the map $\widehat{\delta_\nu} : \R \to \rZ^1(\mathcal R) : t \mapsto \delta_\mu^{{\mathbf i}t}$ uniquely extends to a continuous group homomorphism $\Delta_\nu : G \to \rZ^1(\mathcal R)$ such that $\Delta_\nu(t) = \delta_\mu^{{\mathbf i}t}$ for every $t \in \R$.

For almost every $(x, y) \in \mathcal R$, we have
$$\delta_{\nu_2}(x, y) = \frac{\rd \nu_{2,\ell}}{\rd \nu_{2, r}}(x, y) = \frac{\rd \nu_{2}}{\rd \nu_{1}}(x) \, \frac{\rd \nu_{1,\ell}}{\rd \nu_{1, r}}(x, y) \, \frac{\rd \nu_{1}}{\rd \nu_{2}}(y) = \frac{\rd \nu_{2}}{\rd \nu_{1}}(x) \, \delta_{\nu_1}(x, y) \, \frac{\rd \nu_{1}}{\rd \nu_{2}}(y).$$
This implies that $\Delta_{\nu_2}(t) / \Delta_{\nu_1}(t) = \delta_{\nu_2}^{{\mathbf i}t} / \delta_{\nu_1}^{{\mathbf i}t} \in \rB^1(\mathcal R)$ for every $t \in \R$. By density and since $\rB^1(\mathcal R) \subset \rZ^1(\mathcal R)$ is closed, it follows that $\Delta_{\nu_2}(g) / \Delta_{\nu_1}(g) \in \rB^1(\mathcal R)$ for every $g \in G$.

Since $\mathbf T \subset \rL^0(X, \mathbf T)$ is a closed subgroup and since $\rB^1(\mathcal R) = \rL^0(X, \mathbf T)/\mathbf T$, there exist a Borel map $w : G \to \rL^0(X, \mathbf T)$ such that $\Delta_{\nu_2}(g) / \Delta_{\nu_1}(g) = \partial(w(g))$ for every $g \in G$ and a Borel map $\sigma : G \times G \to \mathbf T$ such that $\sigma \in \rZ^2_{{\rm m}}(G, \mathbf T)$ and $w(gh) = \sigma(g, h) \, w(g)w(h)$ for all $g, h \in G$, where $\rZ^2_{{\rm m}}(G, \mathbf T)$ denotes the group of all measurable scalar $2$-cocycles on $G$. By \cite[Proof of Proposition 5.8]{AM10}, we have $\rZ^2_{{\rm m}}(G, \mathbf T) = \rB^2_{{\rm m}}(G, \mathbf T)$ and hence there exists a measurable map $\upsilon : G \to \mathbf T$ such that $\sigma(g, h) = \upsilon(g) \upsilon(h) \, \overline{\upsilon(gh)}$ for all $g, h \in G$ (see \cite[Theorem 5]{Mo75} for the fact that we may choose $\upsilon : G \to \mathbf T$ so that the $2$-coboundary relation holds everywhere). Define $u : G \to \rL^0(X, \mathbf T) : g \mapsto \upsilon (g) w(g)$. Then $u : G \to \rL^0(X, \mathbf T)$ is a measurable group homomorphism and hence is continuous. Moreover, we have $\Delta_{\nu_2}(g) / \Delta_{\nu_1}(g) = \partial(w(g)) = \partial(u(g))$ for every $g \in G$.

Since $\widehat G = \Lambda$, we may regard the continuous group homomorphism $u : G \to \rL^0(X, \mathbf T)$ as a measurable map $h : X \to \Lambda$ such that $\Delta_{\nu_2}(g) / \Delta_{\nu_1}(g) = \partial(u(g)) = \langle \partial h , g \rangle$ for every $g \in G$. This implies that $\delta_{\nu_2} / \delta_{\nu_1} = \partial h \in  \rB^1(\mathcal R, \Lambda)$.
\end{proof}

\begin{proof}[Proof of Theorem \ref{theorem: Sd invariant}]
Let $\nu \in \mathcal M_{\ap}(X,\cR)$ be any measure and $U \in \mathfrak{P}(X), \; U \neq \emptyset$. If $\cR$ is of type ${\rm I}$ or ${\rm II}$, we have $\Sd(\mathcal R) = \mathbf 1 \subset \Range(\delta_{\nu}|_{\cR_U})$. If $\mathcal R$ is of type ${\rm III}$, there exists $\theta \in [[\cR]]$ such that $\dom(\theta)=U$ and $\ran(\theta)=X$. Then $\theta_\ast \nu \in \mathcal M_{\ap}(X, \mathcal R)$ and $\Sd(\mathcal R) \subset \Range(\delta_{\theta_\ast \nu}) = \Range(\delta_{\nu}|_{\cR_U})$. This shows that in all cases, we have
\begin{equation}\label{equation: inclusion}
\Sd(\mathcal R) \subset \bigcap_{U  \in \mathfrak{P}(X), \; U \neq \emptyset} \Range(\delta_{\nu}|_{\cR_U}).
\end{equation}
On the other hand, we have
$$\bigcap_{\nu \in \mathcal M_{\ap}(X, \mathcal R)} \bigcap_{U  \in \mathfrak{P}(X), \; U \neq \emptyset} \Range(\delta_{\nu}|_{\cR_U}) \subset \Sd(\mathcal R).$$
Therefore, in order to show the reverse inclusion of \eqref{equation: inclusion}, it suffices to prove that for any measures $\nu_1, \nu_2 \in \mathcal M_{\ap}(X, \mathcal R)$, we have
\begin{equation}\label{equation: equality-nu}
\Lambda_1 := \bigcap_{U  \in \mathfrak{P}(X), \; U \neq \emptyset} \Range(\delta_{\nu_1}|_{\cR_U}) = \bigcap_{U  \in \mathfrak{P}(X), \; U \neq \emptyset} \Range(\delta_{\nu_2}|_{\cR_U}) =: \Lambda_2.
\end{equation}
By contradiction, assume that $\Lambda_1 \neq \Lambda_2$. Without loss of generality, we may assume that there exists $\lambda \in \Lambda_1 \setminus \Lambda_2$. By Lemma \ref{lemma: cohomology}, there exists $h \in \rL^{0}(X,\Lambda)$ such that $\delta_{\nu_2}/\delta_{\nu_1} = \partial h$. Since $\lambda \notin \Lambda_2$, there exists a nonzero $U \in \mathfrak{P}(X)$ such that $\lambda \notin \Range(\delta_{\nu_2}|_{\cR_U})$ and up to shrinking $U \in \mathfrak{P}(X)$ if necessary, we may assume that $h|_ 
{U}$ is constant. This means that $\delta_{\nu_1}|_{\cR_U}=\delta_{\nu_2}|_{\cR_U}$. Since $\lambda \in \Lambda_1$, we have $\lambda \in \Range(\delta_{\nu_1}|_{\cR_U})$. 
This however contradicts the fact that $\lambda \notin \Range(\delta_{\nu_2}|_{\cR_U})$.
\end{proof}

We next obtain an analogue of \cite[Theorem 4.7]{Co74}.

\begin{theorem}\label{theorem: uniqueness}
Let $\mathcal R$ be an almost periodic  strongly ergodic equivalence relation on $X$. Put $\Gamma = \Sd(\mathcal R)$. Then there exists a measure $\nu \in \mathcal M_{\ap}(X, \mathcal R)$ that is $\Gamma$-almost periodic.

Moreover, for any $\Gamma$-almost periodic infinite measures $\nu_1, \nu_2 \in \mathcal M_{\ap}(X,\mathcal R)$, there exist $\alpha \in \R^+_0$ and $\theta \in [\mathcal R]$ such that $\theta_\ast \nu_1=\alpha \nu_2$.
\end{theorem}

Before proving Theorem \ref{theorem: uniqueness}, we need some preparation.

\begin{lemma}\label{lemma: equivalence}
Let $\mathcal R$ be an almost periodic strongly ergodic equivalence relation on $X$. Put $\Gamma = \Sd(\mathcal R)$. Let $\nu \in \mathcal M_{\ap}(X,\mathcal R)$ be any measure. The following assertions are equivalent:
\begin{itemize}
\item [$(\rm i)$] $\nu$ is $\Gamma$-almost periodic. 
\item [$(\rm ii)$] The subequivalence relation $\mathcal R_\nu \subset \mathcal R$ defined by $\mathcal R_\nu := \ker(\delta_\nu)$ is ergodic.
\end{itemize}
\end{lemma}

\begin{proof}
$(\rm i) \Rightarrow (\rm ii)$ By Theorem \ref{theorem: Sd invariant}, we have $\Gamma =  \Sd(\mathcal R) \subset \Range({\delta_\nu}|_{\cR_U})$ for every nonzero $U \in \mathfrak{P}(X)$. Since $\nu$ is $\Gamma$-almost periodic, we also have $\Range({\delta_\nu}|_{\cR_U}) \subset \Gamma$ for every nonzero $U \in \mathfrak{P}(X)$. This implies that $\Range({\delta_\nu}|_{\cR_U}) = \Gamma$ for every nonzero $U \in \mathfrak{P}(X)$. In order to show that $\mathcal R_\nu$ is ergodic, it suffices to show that for any nonzero $U, V \in \mathfrak{P}(X)$, there exists a nonzero element $\theta \in [[\cR_\nu]]$ such that $\dom(\theta) \subset U$ and $\ran(\theta) \subset V$. Since $\mathcal R$ is ergodic, we can find a nonzero $\phi \in [[\cR]]$ such that $\dom(\phi) \subset U$, $\ran(\phi) \subset V$ and $\delta_\nu(\phi(y), y) = \lambda \in \Gamma$ for a.e.\ $y \in \dom(\phi)$. Since $\lambda^{-1} \in \Range({\delta_\nu}|_{\cR_{\dom(\phi)}})$, we can find a nonzero $\psi \in [[\cR]]$ such that $\dom(\psi) \subset \dom(\phi)$, $\ran(\psi) \subset \dom(\phi)$ and $\delta_\nu(\psi(x), x) = \lambda^{-1}$ for a.e.\ $x \in \dom(\psi)$. Then $\theta = \phi \psi$ is a nonzero element of $[[\cR_\nu]]$ such that $\dom(\theta) \subset U$ and $\ran(\theta) \subset V$. This shows that $\mathcal R_\nu$ is ergodic.

$(\rm ii) \Rightarrow (\rm i)$ By Theorem \ref{theorem: Sd invariant}, we know that 
$$\Gamma = \bigcap_{U  \in \mathfrak{P}(X), \; U \neq \emptyset} \Range(\delta_{\nu}|_{\cR_U}).$$
Hence, in order to show that $\nu$ is $\Gamma$-almost periodic, it suffices to prove that $\Range(\delta_\nu) = \Range(\delta_{\nu}|_{\cR_U})$ for every nonzero $U \in \mathfrak{P}(X)$. First, we always have $\Range(\delta_{\nu}|_{\cR_U}) \subset \Range(\delta_\nu)$. Next, let $\lambda \in \Range(\delta_\nu)$. Then there exists a nonzero $\phi \in [[\cR]]$ such that $\delta_\nu(\phi(x),x) = \lambda$ for a.e.\ $x \in \dom(\phi)$. Since $\mathcal R_\nu$ is ergodic, we can find a nonzero $\psi_1 \in [[\cR_\nu]]$ with $\dom(\psi_1) \subset U$ and $\ran(\psi_1) \subset \dom(\phi)$ and a nonzero $\psi_2 \in [[\cR_\nu]]$ such that $\dom(\psi_2) \subset \phi(\ran(\psi_1))$ and $\ran(\psi_2) \subset U$. Then $\theta = \psi_2 \phi \psi_1$ is a nonzero element of $[[\cR_U]]$ such that $\delta_\nu(\theta(x),x)=\lambda$ for a.e.\ $x \in \dom(\theta)$. This shows that $\lambda \in \Range(\delta_{\nu}|_{\cR_U})$. Therefore, $\Range(\delta_{\nu}|_{\cR_U}) = \Range(\delta_\nu)$ and hence $\Range(\delta_\nu) = \Gamma$.
\end{proof}

\begin{proof}[Proof of Theorem \ref{theorem: uniqueness}]
First, we prove that there exists a measure $\nu \in \mathcal M_{\ap}(X,\mathcal R)$ that is $\Gamma$-almost periodic. We may assume without loss of generality that $\mathcal R$ is of type ${\rm III}$. Let $\eta \in \mathcal M_{\ap}(X, \mathcal R)$ be any $\Lambda$-almost periodic measure for some countable subgroup $\Lambda < \R^+_0$.  By viewing $\delta_\eta$ as an element of $\rZ^1(\cR,\Lambda)$, we can consider the skew-product equivalence relation $\rd(\mathcal R)=\cR \times_{\delta_\eta} \Lambda$ on $X \times \Lambda$. Observe that the translation action $\Lambda \curvearrowright X \times \Lambda$ acts by automorphisms of the equivalence relation $\rd(\mathcal R)$. Moreover, the equivalence relation generated by $\rd(\mathcal R)$ and $\mathcal R(\Lambda \curvearrowright X \times \Lambda)$ coincides with the equivalence relation $\mathcal S$ defined by 
$$((x, g), (y, h)) \in \mathcal S \quad \text{if and only if} \quad (x, y) \in \mathcal R$$
for a.e.\ $(x,y) \in \cR$ and all $g,h \in \Lambda$. Observe that $\mathcal S$ is nothing but an amplification of $\mathcal R$ and hence $\mathcal S$ is strongly ergodic.

We show that $\rd(\mathcal R)$ has a completely atomic ergodic decomposition. Since the translation action $\Lambda \curvearrowright X \times \Lambda$ acts by automorphisms of the equivalence relation $\rd(\mathcal R)$, it induces an action on $\rL^{\infty}(X \times \Lambda)$ which globally preserves the algebra of $\rd(\cR)$-invariant functions $\rL^{\infty}(X \times \Lambda)^{\rd(\cR)}$. Since $\mathcal S$ is ergodic, the action $\Lambda \curvearrowright \rL^{\infty}(X \times \Lambda)^{\rd(\cR)}$ is ergodic. So, either $\rL^{\infty}(X \times \Lambda)^{\rd(\cR)}$ is discrete (completely atomic) or diffuse. Assume by contradiction that it is diffuse. Since $\Lambda$ is abelian hence amenable, the action $\Lambda \curvearrowright \rL^{\infty}(X \times \Lambda)^{\rd(\cR)}$ is not strongly ergodic (see e.g.\ \cite[Proposition 2.2]{Sc79}). Hence, we can find a non-trivial almost $\Lambda$-invariant sequence in $\rL^{\infty}(X \times \Lambda)^{\rd(\cR)}$. But this provides a non-trivial almost $\mathcal S$-invariant sequence in $\rL^{\infty}(X \times \Lambda)$. This however contradicts the fact that $\mathcal S$ is strongly ergodic.

Since $\mathcal R_\eta \cong \rd(\mathcal R)_{X \times \{1\}}$, it follows that $\mathcal R_\eta$ also has a completely atomic ergodic decomposition. Let $U \in \mathfrak{P}(X)$ be any nonzero subset such that $\mathcal R_{\eta} \cap (U \times U)$ is ergodic. Since $\mathcal R$ is of type ${\rm III}$, there exists $\theta \in [[\cR]]$ such that $\dom(\theta)=U$ and $\ran(\theta)=X$. Put $\nu = \theta_\ast \eta \in \mathcal M_{\ap}(X, \mathcal R)$. Then $\mathcal R_\nu$ is ergodic and hence $\nu$ is $\Gamma$-almost periodic by Lemma \ref{lemma: equivalence}.

Secondly, let $\nu_1, \nu_2 \in \mathcal M_{\ap}(X,  \mathcal R)$ be any $\Gamma$-almost periodic infinite measures. By Lemma \ref{lemma: cohomology}, there exists $h \in \rL^{0}(X,\Gamma)$ such that $\delta_{\nu_2}/\delta_{\nu_1} = \partial h$. For almost every $(x, y) \in \mathcal R$, we have 
$$h(x) \, h(y)^{-1} = \frac{\delta_{\nu_2}(x, y)}{\delta_{\nu_1}(x, y)} = \frac{\rd \nu_2}{\rd \nu_1}(x) \, \frac{\rd \nu_1}{\rd \nu_2}(y)$$
and hence 
$$h(x) \,  \frac{\rd \nu_1}{\rd \nu_2}(x) = h(y) \,  \frac{\rd \nu_1}{\rd \nu_2}(y).$$
Since $\mathcal R$ is ergodic, it follows that the function $h \,  \frac{\rd \nu_1}{\rd \nu_2}$ is constant and equal to some $\alpha \in \R^+_0$. Consider the amplified equivalence relation $\mathcal T=\cR \otimes \mathcal{S}_2$ on $X \times \{0,1\}$. Since $\mathcal{T}$ is isomorphic to $\cR$, we know that $\mathcal T$ is strongly ergodic, almost periodic and that $\Sd(\mathcal T) = \Sd(\mathcal R) = \Gamma$. Define the measure $\nu=\nu_1 \otimes \delta_0+\alpha\nu_2 \otimes \delta_1$ on $X \times \{0, 1\}$. For almost every $(x, y) \in \mathcal R$, we have $\delta_\nu((x, 0), (y, 0)) = \delta_{\nu_1}(x, y)$, $\delta_\nu((x, 1), (y, 1)) = \delta_{\alpha \nu_2}(x, y)$ and $\delta_\nu((x, 0), (y, 1)) = \delta_\nu((x, 0), (x, 1))\, \delta_\nu((x, 1), (y, 1)) = h(x)^{-1} \, \delta_{\alpha \nu_2}(x, y)$. This implies that $\nu$ is $\Gamma$-almost periodic. Hence, by Lemma \ref{lemma: equivalence}, $\mathcal T_\nu$ is ergodic. Since $\nu(X \times \{0\}) = \nu_1(X) = + \infty = \alpha\nu_2(X) = \nu(X \times \{1\})$, there exists a partial isomorphism $\Theta \in [[\mathcal{T}_\nu]]$ such that $\dom(\Theta)=X \times \{0\}$ and $\ran(\Theta)=X \times \{1\}$. Define $\theta \in [\cR]$ by the formula $(\theta(x),1)  = \Theta(x, 0)$ for a.e.\ $x \in X$. Then $\alpha \nu_2 = \theta_\ast\nu_1$. 
\end{proof}

\begin{theorem}\label{theorem: ergodic}
Let $\mathcal R$ be an almost periodic strongly ergodic equivalence relation on $X$. Put $\Gamma = \Sd(\mathcal R)$. Then for any $\Gamma$-almost periodic measure $\nu \in \mathcal M_{\ap}(X,\mathcal R)$, the subequivalence relation $\mathcal R_\nu \subset \mathcal R$ defined by $\mathcal R_\nu := \ker(\delta_\nu)$ is strongly ergodic.
\end{theorem}

\begin{proof}
As in the first part of the proof of Theorem \ref{theorem: uniqueness}, consider the skew-product equivalence relation $\rd(\mathcal R)=\cR \times_{\delta_\nu} \Gamma$ on $X \times \Gamma$. The equivalence relation generated by $\rd(\mathcal R)$ and $\mathcal R(\Gamma \curvearrowright X \times \Gamma)$ coincides with the equivalence relation $\mathcal S$ defined by 
$$((x, g), (y, h)) \in \mathcal S \quad \text{if and only if} \quad (x, y) \in \mathcal R$$
for a.e.\ $(x,y) \in \cR$ and all $g,h \in \Gamma$. Observe that $\mathcal S$ is nothing but an amplification of $\mathcal R$ and hence $\mathcal S$ is strongly ergodic. By Lemma \ref{lemma: equivalence}, we know that $\rd(\mathcal R)_{X \times \{1\}} \cong \mathcal R_\nu$ is ergodic. Since $\Gamma \curvearrowright X \times \Gamma$ preserves $\rd(\mathcal R)$, we have that $\rd(\mathcal R)_{X \times \{\gamma\}}$ is ergodic for every $\gamma \in \Gamma$. It follows that the ergodic component of $X \times \{1 \}$ is of the form $X \times S$ for some subset $S \subset \Gamma$. But, by definition of $\rd(\cR)$, this means that $S=\mathrm{Range}(\delta_\nu)=\Gamma$. Hence $\rd(\cR)$ is ergodic. Now, we conclude, by Proposition \ref{amenable}, that $\rd( \cR)$ is strongly ergodic. Hence $\rd(\mathcal R)_{X \times \{1\}} \cong \mathcal R_\nu$ is also strongly ergodic.
\end{proof}

\begin{proof}[Proof of Theorem \ref{thmstar:almost-periodic}]
$(\rm iii)$ and $(\rm v)$ follow from Theorem \ref{theorem: uniqueness} and $(\rm iv)$ follows from Theorem \ref{theorem: ergodic}. 

$(\rm i)$ Choose $\nu \in \mathcal M_{\ap}(X, \mathcal R)$ as in the first part of Theorem \ref{theorem: uniqueness}. Then $\cR_\nu$ is ergodic by Lemma \ref{lemma: equivalence}. Hence we have $\mathrm{S}(\cR)=\overline{\mathrm{Range}}(\delta_\nu)$. And we also have $\mathrm{Sd}(\cR)= \mathrm{Range}(\delta_\nu)$. Hence $\mathrm{S}(\cR)=\overline{\mathrm{Sd}(\cR)}$.  

$(\rm ii)$ Choose $\nu \in \mathcal M_{\ap}(X,\mathcal R)$ as in the first part of Theorem \ref{theorem: uniqueness}. Let $(t_n)_{n \in \N}$ be any sequence in $\R$. First, assume that $\gamma^{{\bf i}t_n} \to 1$ for every $\gamma \in \Sd(\mathcal R)$. Since $\Range(\delta_\nu) = \Sd(\mathcal R)$, we then have $\delta_\nu^{{\bf i} t_n} \to \mathbf 1$ in $\rZ^1(\mathcal R)$ for the convergence in measure. Conversely, assume that $t_n \to 0$ with respect to $\tau(\mathcal R)$. Then there exists a sequence $u_n \in \rL^0(X, \mathbf T)$ such that $(\partial u_n) \delta_\nu^{{\bf i} t_n} \to \mathbf 1$ in $\rZ^1(\mathcal R)$ for the convergence in measure. Then we have that $\partial u_n \to \mathbf 1$ in $\rZ^1(\mathcal R_\nu)$ for the convergence in measure. Since $\mathcal R_\nu$ is strongly ergodic by Theorem \ref{theorem: ergodic}, there exists 
a sequence $z_n \in \mathbf T$ such that $z_n u_n \to \mathbf 1$ in $\rL^0(X, \mathbf T)$ for the convergence in measure. This implies that $\delta_\nu^{{\bf i} t_n} \to \mathbf 1$ in $\rZ^1(\mathcal R)$ for the convergence in measure. Since $\Sd(\mathcal R) = \Range(\delta_\nu)$, this further implies that $\gamma^{{\bf i}t_n} \to 1$ for every $\gamma \in \Sd(\mathcal R)$.

\end{proof}

%%%%%%%%%%%%%%%%%%%%%%%%%%%%%%%%%%%%%

\section{Explicit computations of $\Sd$ and $\tau$ invariants}\label{section:computations}

\subsection{Equivalence relations arising from actions of bi-exact groups}

Following \cite[Definition 15.1.2]{BO08}, we say that a discrete group $\Gamma$ is {\em bi-exact} if $\Gamma$ is exact and if there exists a map $\mu : \Gamma \to \Prob(\Gamma)$ such that $\lim_{x \to \infty} \|\mu(gxh) - g_\ast \mu(x)\| = 0$ for all $g, h \in \Gamma$. The class of bi-exact discrete groups includes amenable groups, free groups \cite{AO74}, discrete subgroups of simple connected Lie groups of real rank one \cite{Sk88}, Gromov word-hyperbolic groups \cite{Oz03}, wreath product groups $H \wr \Lambda$ where $H$ is amenable and $\Lambda$ is bi-exact \cite{Oz04} and the group $\Z^2 \rtimes \SL_2(\Z)$ \cite{Oz08}.  We refer the reader to \cite[Chapter 15]{BO08} for more information on bi-exact discrete groups.

It was shown in \cite[Theorem C]{HI15} that for any bi-exact countable discrete group $\Gamma$ and  any strongly ergodic free nonsingular action on a standard measure space $\Gamma \curvearrowright (X, \mu)$, the group measure space factor $\rL^\infty(X) \rtimes \Gamma$ is full. The following rigidity result shows that in that case, the Sd and $\tau$ invariants of the orbit equivalence relation $\mathcal R(\Gamma \curvearrowright X)$ coincide with those of the corresponding factor $\rL(\mathcal R(\Gamma \curvearrowright X)) = \rL^\infty(X) \rtimes \Gamma$.

\begin{theorem}[Theorem \ref{thmstar:bi-exact}]
Let $\Gamma$ be any bi-exact countable discrete group and $\Gamma \curvearrowright X$ any strongly ergodic free action on a standard measure space. Then $\rL(\mathcal R(\Gamma \curvearrowright X))$ is a full factor and $$\tau \left(\mathcal R(\Gamma \curvearrowright X) \right) = \tau \left(\rL(\mathcal R(\Gamma \curvearrowright X)) \right).$$
If moreover $\mathcal R(\Gamma \curvearrowright X)$ is almost periodic, then $\rL(\mathcal R(\Gamma \curvearrowright X))$ is almost periodic and 
$$\Sd \left(\mathcal R(\Gamma \curvearrowright X) \right) = \Sd \left(\rL(\mathcal R(\Gamma \curvearrowright X)) \right).$$
\end{theorem}

\begin{proof}
Write $\mathcal R = \mathcal R(\Gamma \curvearrowright X)$, $A = \rL^\infty(X)$ and $M = \rL(\mathcal R)$. Denote by $\rE_A : M \to A$ the unique faithful normal conditional expectation and put $\varphi = \tau_\mu \circ \rE_A \in M_\ast$ where $\mu \in \mathcal{M}(X)$ is any probability measure. By \cite[Theorem C]{HI15}, $M = \rL^\infty(X) \rtimes \Gamma$ is a full factor. 

We first prove that $\tau(\mathcal R) = \tau(M)$. Let $(t_n)_n$ be any sequence in $\R$. It is clear that if $t_n \to 0$ with respect to $\tau(\mathcal R)$ then $t_n \to 0$ with respect to $\tau(M)$. Conversely, assume that $t_n \to 0$ with respect to $\tau(M)$. Then there exists a sequence of unitaries $(u_n)_n$ in $\mathcal U(M)$ such that $\Ad(u_n) \circ \sigma_{t_n}^\varphi \to \id_M$ with respect to the $u$-topology in $\Aut(M)$. Write $\core(M) = M \rtimes_{\sigma^\varphi} \R$ for the continuous core of $M$. Observe that $\core(M) = \rL^\infty(X \times \R) \rtimes \Gamma$ where $\Gamma \curvearrowright X \times \R$ is the Maharam extension of $\Gamma \curvearrowright X$. Denote by $\rE_{\rL^\infty(X \times \R)} : \core(M) \to \rL^\infty(X \times \R)$ the unique faithful normal conditional expectation. Denote by $(\lambda_\varphi(t))_{t \in \R}$ the canonical unitaries in $\core(M)$ implementing the modular flow $\sigma^\varphi$. By \cite[Lemma XII.6.14]{Ta03a}, we have that $\Ad(u_n \lambda_\varphi(t_n)) \to \id_{\core(M)}$ with respect to the $u$-topology in $\Aut(\core(M))$. Since $M$ is a nonamenable factor, $\core(M)$ has no amenable direct summand (see e.g.\ \cite[Proposition 2.8]{BHR12}). Then \cite[Theorem A]{HI15} implies that $\rE_{\rL^\infty(X \times \R)}(u_n \lambda_\varphi(t_n)) - u_n \lambda_\varphi(t_n) \to 0$ with respect to the $\ast$-strong topology. Since $\rL^\infty(X \times \R) = A \ovt \rL(\R)$ (with $\widehat \R = \R$) and $\lambda_\varphi(t_n) \in \mathcal U(\rL(\R))$ for every $n \in \N$, it follows that $\rE_{A}(u_n) - u_n \to 0$ with respect to the $\ast$-strong topology. We can then find a sequence of unitaries $(v_n)_n$ in $\mathcal U(A)$ such that $v_n - u_n \to 0$ with respect to the $\ast$-strong topology. This further implies that $\Ad(v_n) \circ \sigma_{t_n}^\varphi \to \id_M$ with respect to the $u$-topology in $\Aut(M)$. Since $\sigma_t^\varphi |_A = \id_A$ for every $t \in \R$, we infer that $t_n \to 0$ with respect to $\tau(\mathcal R)$ by Lemma \ref{lemma: link}. Therefore, we have $\tau(\mathcal R) = \tau(M)$. 

Next, assume that $\mathcal R$ is almost periodic. Then $M$ is almost periodic as well. It is clear that $\Sd(M) \subset \Sd(\mathcal R)$. By Theorem \ref{theorem: uniqueness}, choose a probability measure $\nu \in \mathcal M_{\ap}(X,\mathcal R)$ that is $\Sd(\mathcal R)$-almost periodic. Put $\psi = \tau_\nu \circ \rE_A \in M_\ast$. By Lemma \ref{lemma: equivalence}, $\mathcal R_\nu$ is ergodic and hence $M_\psi = \rL(\mathcal R_\nu)$ is a factor. Then \cite[Lemma 4.8]{Co74} implies that the point spectrum of $\Delta_\psi$ coincides with $\Sd(M)$. Since $\psi = \tau_\nu \circ \rE_A$ and since $\Range(\delta_\nu) = \Sd(\mathcal R)$, it follows that the point spectrum of $\Delta_\psi$ is $\Sd(\mathcal R)$. Therefore, we have $\Sd(\mathcal R) = \Sd(M)$.
\end{proof}

\subsection{Generalized Bernoulli equivalence relations}

Let $\Lambda$ be any countable discrete group, $I$ any nonempty countable set and $\Lambda \curvearrowright I$ any action. Let $(Y, \eta)$ be any nontrivial standard probability space and $\mathcal S$ any equivalence relation defined on $(Y, \eta)$. Define the product standard probability space $(X, \mu) = (Y, \eta)^I$ and the product equivalence relation $\mathcal S^{\otimes I}$ on $(X, \mu)$ by
\begin{equation*}((x_i)_{i \in I}, (y_i)_{i \in I}) \in \mathcal S^{\otimes I} \quad \Leftrightarrow \quad \exists \mathcal F \subset I \text{ finite subset such that } 
\left\{ 
{\begin{array}{l} \forall i \in \mathcal  F, \quad (x_i, y_i) \in \mathcal S \\ 
\forall i \in I \setminus \mathcal F, \quad x_i = y_i.
\end{array}} \right.
\end{equation*}
Define the {\em generalized Bernoulli action} $\Lambda \curvearrowright X$ by 
$\lambda \cdot (x_i)_{i \in I} = (x_{\lambda^{-1}i })_{i \in I}$ and denote by $\mathcal R(\Lambda \curvearrowright X)$ the corresponding orbit equivalence relation.

\begin{definition}
We define the {\em generalized Bernoulli equivalence relation} $\mathcal R = \mathcal S^{\otimes I} \rtimes  \Lambda$ on $(X,\mu)$ by 
$$((x_i)_{i \in I}, (y_i)_{i \in I}) \in \mathcal R \quad \Leftrightarrow \quad \exists \lambda \in \Lambda \text{ such that } (\lambda \cdot (x_i)_{i \in I}, (y_i)_{i \in I}) \in \mathcal S^{\otimes I}.$$
In other words, $\mathcal R$ is the equivalence relation generated by $\mathcal S^{\otimes I}$ and $\mathcal R(\Lambda \curvearrowright X)$.
\end{definition}

From now on, we assume that the action $\Lambda \curvearrowright I$ has no invariant mean. Then the generalized Bernoulli action $\Lambda \curvearrowright (X, \mu)$ has spectral gap and therefore is strongly ergodic (see e.g.\ \cite[Theorem 1.2]{KT06}). Since $\mathcal R(\Lambda \curvearrowright X) \subset \mathcal R_\mu \subset \mathcal R$, the generalized Bernoulli equivalence relation $\mathcal R$ is strongly ergodic as well as $\mathcal R_\mu := \ker(\delta_\mu)$.

In this subsection, we first show that almost periodic generalized Bernoulli equivalence relations can have prescribed Sd invariant. Let $\Gamma < \R^+_0$ be any nontrivial countable subgroup. Take an enumeration $\Gamma \cap (0, 1) = \{\gamma_n \mid n \geq 1 \}$. Consider $Y = \bigsqcup_{n \geq 1} \{0, 1\}$ and $\eta$ to be the probability measure on $Y$ defined by $\eta = \sum_{n \geq 1}^{\oplus} \frac{1}{2^n(1 + \gamma_n)} (\delta_{0}^{(n)} + \gamma_n \delta_{1}^{(n)})$. Denote by $\mathcal S_2$ the transitive equivalence relation defined on $\{0, 1\}$ and consider the equivalence relation $\mathcal S = \bigsqcup_{n \geq 1} \mathcal S_2^{(n)}$ on the standard measure space $(Y, \eta)$.

\begin{theorem}\label{theorem: Sd prescribed}
Keep the same setup as above. Then the generalized Bernoulli equivalence relation $\mathcal R = \mathcal S^{\otimes I} \rtimes \Lambda$ is almost periodic, strongly ergodic and $\Sd(\mathcal R) = \Gamma$.
\end{theorem}

\begin{proof}
It is clear that the product measure $\mu$ is almost periodic for $\mathcal R$ and that $\Range(\delta_\mu) = \Gamma$. Since $\mathcal R_\mu $ is strongly ergodic, Lemma \ref{lemma: equivalence} shows that $\Sd(\mathcal R) = \Range(\delta_\mu) = \Gamma$.
\end{proof}

We next show that  generalized Bernoulli equivalence relations can have prescribed $\tau$ invariant. Following \cite[Section 5]{Co74}, let $\xi$ be any nonzero finite Borel measure on $\R^+_0$ with a finite first moment, that is, $\int_{\R^+_0} x \, \rd \xi(x) < +\infty$. We normalize $\xi$ so that $\int_{\R^+_0} (1 +  x) \, \rd \xi(x) = 1$. Consider the unitary representation $\rho_\xi : \R \to \mathcal U(\rL^2(\R^+_0, \xi)) $ defined by
$$ (\rho_\xi(t)f)(x)=x^{{\mathbf i}t} f(x) $$
for all $t \in \R$, $x \in \R^{+}_0$ and $f \in \rL^2(\R^+_0, \xi)$. Denote by $\tau(\xi)$ the weakest topology on $\R$ that makes $\rho_\xi$ continuous. Consider $Y = \R^+_0 \times \{0, 1\}$ and $\eta$ the unique probability measure on $Y$ that satisfies
$$ \int_Y f(x, i) \, \rd \eta(x, i) = \int_{\R^+_0} f(x, 0) \, \rd \xi(x) + \int_{\R^+_0} x f(x, 1) \, \rd \xi(x)$$
for any nonnegative Borel function $f : Y \to \R^+$. Denote by $\mathcal S$ the equivalence relation defined on $(Y, \eta)$ by
$$((x_1, i_1), (x_2, i_2)) \in \mathcal S \quad \Leftrightarrow \quad x_1 = x_2.$$

\begin{theorem}
Keep the same setup as above. Then the generalized Bernoulli equivalence relation $\mathcal R = \mathcal S^{\otimes I} \rtimes \Lambda$ is strongly ergodic and $\tau(\mathcal R) = \tau(\xi)$.
\end{theorem}

\begin{proof}
For every $x \in \R^+_0$, we have $\delta_\eta((x, 1), (x, 0)) = x$. This implies that the weakest topology on $\mathbf R$ that makes the map $\R \to \rZ^1(\mathcal S) : t \mapsto \delta_{\eta}^{{\mathbf i}t}$ continuous is $\tau(\xi)$. By construction of $\mathcal R = \mathcal S^{\otimes I} \rtimes \Lambda$, it follows that the weakest topology on $\mathbf R$ that makes the map $\R \to \rZ^1(\mathcal R) : t \mapsto \delta_{\mu}^{{\mathbf i}t}$ continuous is $\tau(\xi)$ as well. In order to prove that $\tau(\mathcal R) = \tau(\xi)$, it suffices to show that for any sequence $(t_n)_n$  in $\R$ such that $t_n \to 0$ with respect to $\tau(\mathcal R)$, we have $\delta_\mu^{{\mathbf i}t_n} \to \mathbf 1$ in $\rZ^1(\mathcal R)$ for the convergence in measure. Assume that $t_n \to 0$ with respect to $\tau(\mathcal R)$. Then there exists a sequence $u_n \in \rL^0(X, \T)$ such that $(\partial u_n)  \delta_\mu^{{\mathbf i}t_n} \to \mathbf 1$ in $\rZ^1(\mathcal R)$ for the convergence in measure. In particular, we have $\partial u_n  \to \mathbf 1$ in $\rZ^1(\mathcal R_\mu)$ for the convergence in measure. Since $\mathcal R_\mu$ is strongly ergodic, there exists a sequence $z_n \in \T$ such that $ z_n c_n \to \mathbf 1$ in $\rL^0(X, \T)$ for the convergence in measure. This finally shows that $\delta_\mu^{{\mathbf i}t_n} \to \mathbf 1$ in $\rZ^1(\mathcal R)$ for the convergence in measure.
\end{proof}

Finally, we point out that we can construct examples of strongly ergodic generalized Bernoulli equivalences $\mathcal R$ with prescribed Sd and $\tau$ invariants for which the associated factor $\rL(\mathcal R)$ is not full. For this, we need the following group theoretic result.

\begin{proposition}\label{proposition: group-action}
Write $\Z_2 = \Z/2\Z$ and $\Z_3 = \Z/3\Z$. Put $\Gamma = \Z_2 \ast \Z_3$,  $\Lambda = \Gamma^{\oplus \N}$, $H = \Z_2^{\oplus \N}$ and $I = \Lambda /H$. Then for every $\lambda \in \Lambda \setminus \{1\}$, there are infinitely many $i \in I$ such that $\lambda \cdot i \neq i$ and the action $\Lambda \curvearrowright I$ has no invariant mean. 
\end{proposition}

\begin{proof}
Observe that for every $g \in \Gamma \setminus \{1\}$, the conjugacy class $\{\gamma g \gamma^{-1} : \gamma \in \Gamma\}$ is infinite. This implies that for every $\lambda \in \Lambda \setminus \{1\}$, there are infinitely many $i \in I$ such that $\lambda \cdot i \neq i$. Since $\Lambda$ is nonamenable and $H$ is amenable, the action $\Lambda \curvearrowright I$ has no invariant mean.
\end{proof}

\begin{corollary}
Let $\Lambda \curvearrowright I$ be any action as in Proposition \ref{proposition: group-action}. Then for any nontrivial standard probability space $(Y, \eta)$, the generalized Bernoulli equivalence relation $\mathcal R = \mathcal S^{\otimes I} \rtimes \Lambda$ is strongly ergodic and the corresponding factor $\rL(\mathcal R)$ is not full.
\end{corollary}

\begin{proof}
We already know that the generalized Bernoulli equivalence relation $\mathcal R = \mathcal S^{\otimes I} \rtimes \Lambda$ is strongly ergodic. Since for every $\lambda \in \Lambda \setminus \{1\}$, there are infinitely many $i \in I$ such that $\lambda \cdot i \neq i$, the action $\Lambda \curvearrowright X$ is essentially free. We then have the crossed product decomposition $\rL(\mathcal R) = \rL(\mathcal S)^I \rtimes \Lambda$. Write $\Z_2 = \{1, s\}$. For every $n \in \N$, let $h_n = 1 \oplus \cdots \oplus 1 \oplus s \in H$ where the element $s$ sits at the $n$th position. Then the sequence $(h_n)_n$ is eventually central in $\Lambda$. Thus,  the sequence $(u_{h_n})_n$ is a nontrivial uniformly bounded central sequence in the group von Neumann algebra $\rL(\Lambda)$ such that for every $i \in I$, we have $u_{h_n} \in \rL(\Stab(i))$ eventually. Then \cite[Lemma 2.7]{VV14} implies that $\rL(\mathcal R)$ is not full.
\end{proof}

\begin{remark}\label{remark: plain}
We point out that for plain Bernoulli equivalence relations $\mathcal R = \mathcal S^{\otimes \Lambda} \rtimes \Lambda$ arising from nonamenable groups $\Lambda$, the corresponding factor  $\rL(\mathcal R)$ is full and $\tau(\mathcal R) = \tau(\rL(\mathcal R))$ (see \cite[Lemma 2.7]{VV14}). If $\mathcal S$ is moreover as in Theorem \ref{theorem: Sd prescribed}, then we have $\Sd(\mathcal R) = \Sd(\rL(\mathcal R))$.
\end{remark}

\appendix

\section{Bounded elements in the full pseudo-group}

\begin{theorem} \label{bounded_typeIII}
Let $(X,\mu)$ be a standard probability space and $\mathcal{R} \subset X \times X$ an ergodic equivalence relation of type $\III_\lambda$ with $0 <\lambda \leq 1$. Then for any nonzero  subsets $A,B \in \mathfrak{P}(X)$, we can find a $\mu$-bounded partial isomorphism $\theta : A \rightarrow B$ in $[[\mathcal{R}]]$.  More precisely, if $0< C < \lambda$, we can choose $\theta$ so that for a.e.\ $x \in A$ we have
\[ C \frac{\mu(B)}{\mu(A)} \leq \frac{\rd (\mu \circ \theta)}{\rd \mu}(x) \leq C^{-1} \frac{\mu(B)}{\mu(A)}.  \]
\end{theorem}

The proof relies on a maximality argument based on the following lemma.

\begin{lemma} \label{bounded_small}
Let $(X,\mu)$ be a standard probability space and $\mathcal{R} \subset X \times X$ an ergodic equivalence relation of type $\III_\lambda$ with $0 <\lambda \leq 1$. Let $A,B \in \mathfrak{P}(X)$ be two nonzero subsets and $c_1,c_2 > 0$ two constants such that $ c_1 < \lambda c_2$. Then we can find a nonzero element $\pi \in [[\cR]]$ such that $\dom(\pi) \subset A$, $\ran(\pi) \subset B$ and
\[ c_1 \leq \frac{\rd (\mu \circ \pi)}{\rd \mu}(x) \leq c_2 \]
for a.e.\ $x \in \dom(\pi)$.
\end{lemma}
\begin{proof}
First we can find a nonzero $\psi \in [[\cR]]$ such that $U:=\dom(\psi) \subset A$ and $\ran(\psi) \subset B$. Moreover, by shrinking $\psi$ if necessary, we can assume that
\[ b \leq \frac{\rd (\mu \circ \psi)}{\rd \mu}(x) \leq b \sqrt{\lambda c_2/c_1} \]
for some constant $b > 0$ and for a.e.\ $x \in \dom(\psi)$.

Since $c_1/b < \lambda c_2 \left(b\sqrt{\lambda c_2/c_1}\right)^{-1}$ we know that
\[ \mathrm{S}(\cR) \cap \left[ c_1/b, c_2 \left(b\sqrt{\lambda c_2/c_1}\right)^{-1} \right] \neq \emptyset. \]
But we also know that
\[ \mathrm{S}(\cR) \subset \overline{\mathrm{Range}}({\delta_\mu}|_{\cR_{U}}). \]
Hence, we obtain that 
\[ \delta_\mu^{-1}\left( \left[ c_1/b, c_2 \left(b\sqrt{\lambda c_2/c_1}\right)^{-1} \right] \right) \in \mathfrak{P}(\cR) \]
has a non-trivial intersection with $U \times U$. Therefore, we can find a nonzero $\phi \in [[\cR]]$ such that $\dom(\phi),\ran(\phi) \subset U$ and
\[ c_1/b \leq \frac{\rd (\mu \circ \phi)}{\rd \mu}(x) \leq c_2 \left(b\sqrt{\lambda c_2/c_1}\right)^{-1} \]
for a.e.\ $x \in \dom(\phi)$. Hence $\pi=\psi \phi \in [[\cR]]$ is a nonzero element that satisfies the condition we wanted.
\end{proof}

\begin{proof}[Proof of Theorem \ref{bounded_typeIII}]
Take any constant $c \in ]\lambda^{-1}C,1 [$. Consider the set $\Omega$ of all elements $\theta \in [[\mathcal{R}]]$ such that for a.e.\ $x \in \dom(\theta)$ we have
\[ C\frac{\mu(B)}{\mu(A)}   \leq \frac{\rd (\mu \circ \theta)}{\rd \mu}(x)  \leq C^{-1}\frac{\mu(B)}{\mu(A)}  \]
and 
\[c \frac{\mu(B)}{\mu(A)}\mu(A \setminus \dom(\theta))  \leq \mu(B \setminus \ran(\theta)) \leq c^{-1}\frac{\mu(B)}{\mu(A)}\mu(A \setminus \dom(\theta)). \]
The poset $\Omega$ is closed in $[[\mathcal{R}]]$ hence it is inductive. Let $\theta \in \Omega$ be a maximal element and let us show that $\mu(A \setminus \dom(\theta)) = 0$ and $\mu(B \setminus \ran(\theta)) = 0$.

Suppose that \[c \frac{\mu(B)}{\mu(A)}\mu(A \setminus \dom(\theta))  < \mu(B \setminus \ran(\theta)). \] 
Then we will contradict the maximality of $\theta$. Since $\mathcal{R}$ is of type $\III_\lambda$ and $c^{-1} < \lambda C^{-1}$, then by Lemma \ref{bounded_small}, we can find a nonzero $\pi \in [[\mathcal{R}]]$ with $\dom(\pi) \subset A \setminus \dom(\theta)$ and $\ran(\pi) \subset B \setminus \ran(\theta)$ such that for a.e.\ $x \in \dom(\pi)$ we have
\[ c^{-1}\frac{\mu(B)}{\mu(A)}  \leq \frac{\rd (\mu \circ \pi)}{\rd \mu}(x)   \leq C^{-1} \frac{\mu(B)}{\mu(A)}.\]

Moreover, since \[c \frac{\mu(B)}{\mu(A)}\mu(A \setminus \dom(\theta))  < \mu(B \setminus \ran(\theta)), \] we can choose $\pi$ to be small enough so that \[c \frac{\mu(B)}{\mu(A)}\mu(A \setminus \dom(\theta'))  \leq \mu(B \setminus \ran(\theta')), \] 
 where $\theta'=\theta+\pi$. Then it is easy to check, by construction, that $\theta' \in \Omega$ and this contradicts the maximality of $\theta$. Hence we must have \[c \frac{\mu(B)}{\mu(A)}\mu(A \setminus \dom(\theta)) = \mu(B \setminus \ran(\theta)). \]  Similarly, by exchanging the roles of $\theta$ and $\theta^{-1}$, we can show that \[ \mu(B \setminus \ran(\theta)) =c^{-1} \frac{\mu(B)}{\mu(A)}\mu(A \setminus \dom(\theta)). \]  Since $c < 1$, this implies that \[ \mu(B \setminus \ran(\theta)) = \mu(A \setminus \dom(\theta)) = 0, \] as we wanted.
\end{proof}

\begin{corollary}\label{cor-dense}
Let $(X,\mu)$ be a standard probability space and $\mathcal{R} \subset X \times X$ an ergodic equivalence relation of type $\III_\lambda$ with $0 <\lambda \leq 1$. Then the $\mu$-bounded elements are dense in $[\mathcal{R}]$.
\end{corollary}
\begin{proof}
Take $\theta \in [\cR]$. Take $A$ such that the partial isomorphism $\theta|_{X \setminus A}$ is $\mu$-bounded and let $B:=\theta(A)$. By Theorem \ref{bounded_typeIII}, we can find a partial isomorphism $\pi : A \rightarrow B$ which is $\mu$-bounded. Then $\theta'=\theta|_{X \setminus A} + \pi \in [\cR]$ is $\mu$-bounded. Since we can choose $A$ and $B=\theta(A)$  to be arbitrarily small, we see that $\theta'$ is arbitrarily close to $\theta$. Hence the $\mu$-bounded elements are dense in $[\cR]$.
\end{proof}

\begin{corollary}
Let $(X,\mu)$ be a standard probability space and $\mathcal{R} \subset X \times X$ an ergodic equivalence relation of type $\III_\lambda$ with $0 <\lambda \leq 1$. If $\nu$ is another probability measure on $X$ which is equivalent to $\mu$ then for any $C < \lambda$, we can find $\theta \in [\mathcal{R}]$ such that
\[ C \mu \leq \nu \circ \theta \leq C^{-1}\mu. \]
\end{corollary}
\begin{proof}
Let $\mathcal{S}_2$ be the transitive equivalence relation on the set $\{0,1 \}$ and consider the product equivalence relation $\cR \otimes \mathcal{S}_2$ on $X \times \{0,1 \}$. On $X \times \{0,1 \}$, put the measure $\omega=\frac{1}{2}(\mu \otimes \delta_0 + \nu \otimes \delta_1)$. Let $A= X \times \{0\}$ and $B=X \times \{1 \}$. Then $\omega(A)=\omega(B)=\frac{1}{2}$. Thus, Theorem \ref{bounded_typeIII} provides us with a partial isomorphism $\theta : A \rightarrow B$ in $[[ \mathcal{R} \otimes \mathcal{S}_2 ]]$ such that 
\[ C \leq \frac{\rd (\omega \circ \theta)}{\rd \omega}(x)  \leq C^{-1} \]
for a.e.\ $x \in A$. Viewing $\theta$ as an element of $[ \cR ]$, we obtain
\begin{equation*}
 C \leq \frac{\rd ( \nu \circ \theta )}{\rd \mu} \leq C^{-1}. \qedhere
\end{equation*}
\end{proof}

\bibliographystyle{plain}

\end{document}